\documentclass[12pt]{amsart}
\usepackage{amsthm}
\usepackage{amsmath}
\usepackage{amssymb}
\usepackage{mathtools}
\usepackage{comment}
\usepackage{mathrsfs}
\usepackage{hyperref}
\usepackage{xcolor}
\usepackage{inputenc}
\usepackage{stackengine} 
\usepackage{multicol}
\usepackage[misc]{ifsym}
\usepackage{fullpage}

\theoremstyle{plain}
\newtheorem{theorem}{Theorem}
\newtheorem*{theorem*}{Theorem}
\newtheorem{lemma}[theorem]{Lemma}
\newtheorem{proposition}[theorem]{Proposition}
\newtheorem{corollary}[theorem]{Corollary}
\theoremstyle{definition}
\newtheorem{definition}[theorem]{Definition}
\newtheorem{example}[theorem]{Example}
\newtheorem{remark}[theorem]{Remark}
\numberwithin{theorem}{section}
\numberwithin{equation}{section}

\newcommand{\B}{\mathbb{B}}
\newcommand{\C}{\mathbb{C}}
\newcommand{\N}{\mathbb{N}}

\newcommand{\Q}{\mathbb{Q}}
\newcommand{\R}{\mathbb{R}}
\newcommand{\T}{\mathbb{T}}

\newcommand{\Z}{\mathbb{Z}}
\newcommand{\G}{G}

\newcommand{\fa}{\mathfrak{a}}

\newcommand{\fA}{\mathfrak{A}}

\newcommand{\cH}{\mathcal{H}}
\newcommand{\cT}{\mathcal{T}}
\newcommand{\cA}{\mathcal{A}}
\newcommand{\cO}{\mathcal{O}}
\newcommand{\cP}{\mathcal{P}}
\newcommand{\cS}{\mathcal{S}}
\newcommand{\cF}{\mathcal{F}}
\newcommand{\cB}{\mathcal{B}}
\newcommand{\cL}{\mathcal{L}}

\newcommand{\rB}{\mathrm{B}}
\newcommand{\rI}{\mathrm{I}}
\newcommand{\rL}{\mathrm{L}}
\newcommand{\rM}{\mathrm{M}}

\newcommand{\rP}{\mathrm{P}}
\newcommand{\rS}{\mathrm{S}}
\newcommand{\rU}{\mathrm{U}}

\newcommand{\SU}{\mathrm{SU}}
\newcommand{\SUn}{G}
\newcommand{\wSUn}{\widetilde{G}}
\newcommand{\U}{\mathrm{U}}

\newcommand{\sun}{\mathfrak{su}(n,1)}

\newcommand{\tr}{\mathrm{tr}}
\newcommand{\im}{\mathrm{Im}}

\newcommand{\ip}[2]{\langle #1,#2 \rangle}

\newcommand{\vm}{\textbf{m}}

\newcommand{\wG}{\widetilde{G}}
\newcommand{\bcH}{\overline{\cH}}

\newcommand{\Berg}{\mathcal{A}^2_\nu}
\newcommand{\ipnu}[2]{\ip{#1}{#2}_\nu}
\newcommand{\kernu}[1]{k^{(\nu)}_{#1}}
\newcommand{\Kernu}[1]{K^{(\nu)}_{#1}}
\newcommand{\Berz}{B^{(\nu)}}
\newcommand{\Toep}{T^{(\nu)}}
\newcommand{\hnu}{h_\nu}

\newcommand{\LoneG}{L^1(\G)}
\newcommand{\LpG}{L^p(\G)}
\newcommand{\bddfG}{L^\infty(\G)}
\newcommand{\LoneKGK}{L^1(K\backslash G / K)}

\newcommand{\Lpinv}{L^p(\B^n, d\lambda)}

\newcommand{\LpinvK}{L^p(\B^n, d\lambda)^K}
\newcommand{\LoneinvK}{L^1(\B^n, d\lambda)^K}

\newcommand{\nball}{\mathbb{B}^n}
\newcommand{\bdd}{\cL(\cA^2_\nu)}
\newcommand{\traceop}{\cT^1(\cA^2_\nu)}
\newcommand{\schatten}{\cT^p(\cA^2_\nu)}

\newcommand{\bddK}{\cL(\cA^2_\nu)^K}
\newcommand{\traceopK}{\cT^1(\cA^2_\nu)^K}
\newcommand{\haar}[1]{\ d\mu_\G(#1)}
\newcommand{\haarK}[1]{\ d\mu_K(#1)}
\newcommand{\inv}[1]{\ d\lambda(#1)}
\newcommand{\leb}[1]{\ d#1}
\newcommand{\planch}[1]{\ d\mu_P(#1)}
\newcommand{\bmeas}[1]{\ d\mu_\nu(#1)}

\newcommand{\actfL}[2]{\ell_{#1} #2}
\newcommand{\actopL}[2]{L^{\nu}_{#1}(#2)}

\newcommand{\convolalgoplus}{\LoneinvK \oplus \traceopK}
\newcommand{\convolalg}{\mathfrak{A}_{\nu}}

\newcommand{\maxid}{\mathcal M(\convolalg)}

\newcommand{\spheri}{\cS_{bdd}}
\newcommand{\spheripd}{\cS_{pd}}

\newcommand{\pdsf}{\cS_{pd}(G/K)}
\newcommand{\bddspheri}{\cS_{bdd}(G/K)}
\newcommand{\spf}{\cS (G/K)}

\newcommand{\Tr}{\mathrm{Tr}\, }

\newcommand{\rad}{\mathrm{rad}}
\newcommand{\radop}[1]{\mathrm{rad}(#1)}

\newcommand{\id}{\mathrm{id}}

\newcommand{\hgf}{{}_2 F_1}

\DeclareMathOperator{\diag}{diag}

\newcommand{\prma}{P_m}

\newcommand{\wf}{\widehat{f}}

 \newcommand{\note}[2][\null]{%
   \marginpar{\renewcommand{\baselinestretch}{1}\vspace{-1em}\hrule\vspace{3pt}%
  \raggedright\textsf{ \footnotesize#2\ifx \footnotesize#1\null\else\\\hfill--- 
   {\em #1}\fi}\vspace{1.5em}}%
 }

\newcommand{\ds}{\displaystyle}
\newcommand{\pinu}{\pi_\nu}

\newcommand{\HS}{\mathcal{HS}}

\newcommand{\bk}{\mathbf{k}}
\newcommand{\bm}{\mathbf{m}}
\newcommand{\bn}{\mathbf{n}}

\newcommand{\Aut}{\mathrm{Aut}}

\title{The function-operator convolution algebra over the Bergman space of the ball and its Gelfand theory}
\author{Vishwa Dewage, Robert Fulsche and Gestur \'{O}lafsson}

\begin{document}
\begin{abstract}
    We investigate the structure of the commutative Banach algebra formed as the direct sum of integrable radial functions on the disc and the radial operators on the Bergman space, endowed with the convolution from quantum harmonic analysis as the product. In particular, we study the Gelfand theory of this algebra and discuss certain properties of the appropriate Fourier transform of operators which naturally arises from the Gelfand transform.
\end{abstract}
\maketitle

\setcounter{tocdepth}{1}
\tableofcontents


\section{Introduction}

The theory of \emph{quantum harmonic analysis}, developed initially by Reinhard F.~Werner in \cite{Werner1984}, has in the last few years been the center of a significant line of research. Besides studying the structures that arise in this theory out of intrinsic interest \cite{Berge_Berge_Fulsche2024, Fulsche_Luef_Werner2026, Kiukas_etal_2012, Luef_Samuelsen2025}, there has also been a broad range of applications in time-frequency analysis \cite{Luef_Skrettingland2018, Dorfler_Luef_Skrettingland2024, Luef_Skrettingland2020, Dorfler_etal2024} as well as operator theory \cite{Fulsche2019, Fulsche_Rodriguez2025, Dewage_Mitkovski2025, D26} - to name just a few of the many references.

Reinhard Werner considered his initial concepts of quantum harmonic analysis (QHA) with respect to a certain projective unitary representation of $\mathbb R^{2d}$. Given the success of QHA in terms of its applications, there has been significant efforts to extend the notions and tools from QHA to the setting of representations of more general groups \cite{Berge_etal2022, Berge_Halvdansson2025, DDMO25, Fulsche_Galke2025, Halvdansson2023}. 

At the heart of Werner's theory of quantum harmonic analysis lies the possibility of endowing the 
space $L^1(\mathbb R^{2d}) \oplus \mathcal T^1(L^2(\mathbb R^d))$, the direct sum of the integrable 
functions on $\mathbb R^{2d}$ with the trace class operators over 
$L^2(\mathbb R^d)$, with a multiplication, which turns the space into a commutative Banach algebra and extends the 
classical notion of convolution on $L^1(\mathbb R^{2d})$. This multiplication is often referred to as 
the \emph{function-operator convolution}. This product respects the natural grading of the direct sum, i.e., the convolution 
of a function and an operator yields an operator, whereas the convolution of two operators gives a function. The 
structure of this Banach algebra is really the starting point for many of the applications in QHA. Nevertheless, also the structure 
of this Banach algebra itself has been investigated \cite{Berge_Berge_Fulsche2024}.

Let $K$ denote the group $\mathrm{U}_n$ of unitary $n \times n$-matrices, understood as a  maximal compact subgroup of 
$G=\mathrm{SU}(n,1)$. Then, the homogeneous space $G / K$ can be identified with the unit ball $\mathbb B^n$ in $\mathbb C^n$. 
Further, we denote by $\Berg$ the Bergman space over $\mathbb B^n$ with respect to the standard weight parameter $\nu$. The 
first steps towards generalizing the methods and ideas of QHA to this settings was taken in \cite{DDMO25} by consider convolution
of operators and convolutions of functions with operators.  Here we take those ideas further   by
  considering in a systematic way the algebra  $\LoneinvK \oplus \traceopK$, where $\lambda$ is the M\"{o}bius-invariant measure and $\traceopK$ denotes the $K$-invariant, or radial, trace class operators on $\Berg$. Here the product is, as in Werner's case, includes both
 the convolution of functions, convolution of operators and the mixed convolution. It turns out that this leads to  a commutative Banach algebra.
This observation is the starting point of the present work.

The goal of the present work is two-fold: First, we determine the Gelfand spectrum of the commutative Banach algebra 
$\mathfrak A_\nu = \LoneinvK \oplus \traceopK$  with and without the natural 
$*$-involution $(f,A)^* = (f^*,A^*)$.  Since it is well-known that the Gelfand theory of $\LoneKGK$ can be 
described in terms of the space $\bddspheri$ of bounded spherical functions with respect to the Gelfand pair 
$(G, K)$, and taking into account that $\LoneinvK \cong \LoneKGK$ is contained as a subalgebra in $\mathfrak A_\nu$, it 
is no surprise that $\bddspheri$ in one way or another shows up in the description of the Gelfand spectrum of $\mathfrak A_\nu$. Indeed, we prove the following:
\begin{theorem*}
    The Gelfand spectrum $\mathcal M(\mathfrak A_\nu)$ can be described as:
    \begin{align*}
        \mathcal M(\mathfrak A_\nu) \cong \bddspheri \times \mathbb Z_2.
    \end{align*}
\end{theorem*}

Similarly using that the space of $*$-homomorphism of $L^1(\nball )^K$ is the space $\pdsf$ of positive definite spherical functions we show that:
\begin{theorem*}
The space of $*$-homomorphism of $\fA_\nu$ is  isomorphic to $\pdsf \times \Z_2$. \end{theorem*}

Having obtained this result, we then turn to the description of the appropriate notion of a \emph{Fourier transform} for operators in $\traceopK$, 
as motivated by the Fourier transform for operators in Werner's original work.  Here again we use the positive definite functions to be
able to derive a Plancherel formula.
We define the notion of a Fourier transform  $A\mapsto \widehat{A}$, the spherical Fourier transform of $A \in \traceopK$, which yields a function on $\pdsf$. This transformation then enjoys several properties analogous to the usual spherical transform of functions in $\LoneKGK$. It is the second goal of this paper to describe and prove these properties.  We refer to Section \ref{sec:FTop} for more details.

The paper is organized as follows: We start with Section \ref{sec:prelims} on Preliminaries, which is supposed to set up stage for everything that follows. In Section \ref{sec:L1K}, we recall several important facts regarding the Banach algebra $\LoneKGK$. Section \ref{sec:QHA} serves as a repetition of some facts from \cite{DDMO25} (and also to obtain some new results) to set up the commutative Banach algebra $\LoneinvK \oplus \traceopK$. The following Section \ref{sec:convolalg} is then dedicated to describing the structure theory, i.e., the Gelfand spectrum, of this algebra. In Section \ref{sec:plancherel} we discuss the Plancherel theorem for our setting and in the final Section \ref{sec:fouriertrafo}, we then discuss several other results regarding the Fourier transform of operators. The paper is amended by an appendix \ref{app:Gelfandpairs} on facts regarding Gelfand pairs , which will be used frequently throughout this paper.


\section{Preliminaries}\label{sec:prelims}

\noindent
In this section, we introduce the Bergman spaces on the unit ball $\B^n$ in $\C^n$ and give the main theorems that
are needed in what follows. In particular, we discuss the reproducing kernel and the Bergman projection. We then
discuss the holomorphic discrete series of the group $G=\SU (n,1)$ acting on the Bergman space. We start by introducing
some standard facts that will be used later.

\subsection{Bergman spaces on the unit ball}
The unit ball $\nball$ in $\C^n$ is given by
\[\nball = \{z\in \C^n\mid |z| <1\}\]
where $|z| =\left(\sum_{j=1}^n z_j\overline{z_j}\right)^{1/2}$ is the   norm associated to
the inner product 
$\ip{w}{z} = w\cdot\overline{z}
=w_1\overline{z}_1+\dots+w_n\overline{z}_n$ and $z\cdot w = (z,w) = z_1w_1+ \ldots + z_nw_n$ the canonical
bilinear form on $\C^n$.
We identify $\C^n $ with  $\R^{2n}$ and in that identification $\nball$  is isomorphic to the unit ball in $\R^{2n}$.
We denote by $\rS^{2n-1}$ be the boundary of $\nball$, i.e., the unit sphere, and denote by $\sigma_n$ the unique $\SU (n)$ invariant
probability measure on $\rS^{2n-1}$. The {\it polar coordinates} on $\nball$ are given by $(0,1)\times \rS^{2n-1}\to \nball \setminus \{0\}$, 
$(r,\omega) \mapsto r\omega$. We denote the normalized Lebesgue measure on $\B^n$ by  $dz=2n r^{2n-1}drd\sigma_n(w)$, $z=rw$.
 
For $\nu\in \R$ define  
\begin{equation}\label{defMea}
\bmeas{z} = c_\nu (1-|z|^2)^{\nu-(n+1)} \leb{z}.
\end{equation}
This measure
is finite if and only if $\nu >n$ and in that case we take  
$c_\nu = \frac{\Gamma(\nu)}{n!\, \Gamma(\nu-n)}$ so
that $\mu_\nu$ is a probability measure. For $\nu>  n $ we denote by $L_\nu^2(\nball)$ the weighted Lebesgue space
defined by the norm
\begin{equation*}
  \| f\|_{ \nu} = \left( \int_{\B^n} |f(z)|^2 \,d\mu_\nu(z)\right)^{1/2}.
\end{equation*}
But we note that in general $L^2(\nball ,(1-|z|^2)^{\nu - (n+1)}dz)$ is well defined for $\nu \in\R$ and this will be used later.
From now on, we will always assume that $\nu >n$, unless clearly stated otherwise. 

Denote by $\cO (\nball )$ the Fr\'echet space of holomorphic functions on the $n$-ball. Then  the {\it weighted Bergman space} is  
\[ \Berg = \cO (\nball )\cap L^2_\nu(\nball ), \quad \nu >n.\]
 The space $\Berg$ is a reproducing kernel Hilbert space with the $L^2_\nu$ inner product $\langle\cdot, \cdot\rangle_\nu$.
  
\begin{remark}
    Sometimes the Bergman spaces are parametrized by $\alpha= \nu -n-1>-1$ instead of $\nu$, see \cite{Z05}.
\end{remark}
 
For a multi-index $\vm=(m_1,\ldots , m_n)$ let 
$\vm! = m_1!\cdots m_n!$, $|\vm| =m_1+\ldots + m_n$,
$z^\vm = z_1^{m_!}\cdots z_n^{m_n}$ and
\[c(\vm,\nu,n)= c(\vm , \nu ) = \frac{\vm!\, \Gamma (\nu)}{\Gamma (\nu+|\vm |)}.\]
Finally, we let $h (z,w) =1-\ip{z}{w}$ and $\hnu$
\begin{align}\label{eq:hnu}
    \hnu(z)=h(z,z)^{\nu}=(1-|z|^2)^\nu, \quad z\in \B^n.
\end{align}

 We collect here the main facts about the Hilbert space $\Berg$. The proofs can be found in \cite{Z05}:
 \begin{theorem}\label{lem:ProAa}  The following holds:
 \begin{enumerate}
 \item The space of polynomials $\rP(\C^n )$ is dense in $\Berg$.
 \item For $\vm \in \N_0^n$ we have $\|z^\vm \|_\nu^2 = c(\vm,\nu)$.
 \item The normalized monomials
 \[f_\vm (z)=c(\vm,\nu)^{-1/2} z^\vm\, ,\quad
\vm\in\N_0^n\, \]
form an orthonormal basis for $\Berg$.
\item The reproducing kernel for the space $\Berg$ is given by
 \[K_\nu (z,w) = h(z,w)^{-\nu}=\sum_{\vm\in\N_0^n} f_\vm(z) \overline{f_\vm (w)}, \quad z,w\in \nball . \] 
 \item For $w\in \nball $ let $\Kernu{w} (z) = K_\nu (z,w)$. Then $\Kernu{w}\in \Berg$ with:
 \[\|\Kernu{w}\|^2_\nu =
 K_{\nu} (w , w) = (1-|w|^2)^{-\nu } = h(w)^{-\nu} .\]
 \item The  orthogonal  projection $\rP^{(\nu)} : \rL^2_\nu(\nball )\to \Berg$, the Bergman projection, is given by
 \[\rP^{(\nu)} (f) (z)=\ipnu{f}{\Kernu{z}} =\int_{\nball } f(w) K_\nu(z,w) \bmeas{w}, \quad z\in \B^n, f\in \rL^2_\nu(\nball ).\]
 \end{enumerate}
 \end{theorem}

 The {\it normalized reproducing kernel} $k_\nu$ is given by
 \begin{align}\label{eq:kz}
 k_\nu (z,w) &= \kernu{w}(z) = \frac{\Kernu{w}(z)}{\|\Kernu{w}\|_\nu }\nonumber\\
  &=(1-|w|^2)^{\frac{\nu}{2}}\sum_\vm
 f_\vm(z)\overline{f_\vm (w)}\\
 &=\frac{(1-|w|^2)^{\nu / 2}}{ (1-\ip{z}{w})^{\nu } }.\nonumber
\end{align}

\subsection{Toeplitz operators and Berezin transforms}
 
 We define the {\it Berezin transform} $\Berz (S) : \nball \to   C(\C)$ of a bounded operator $S\in \bdd$ by
\begin{align}
\Berz(S)(z) & = \ip{S\kernu{z}}{\kernu{z}}_\nu = (1-|z|^2)^{ \nu}
\ipnu{S \Kernu{z}}{\Kernu{z}}\label{def:BS}\\
&=(1-|z|^2)^{ \nu}\int (S\Kernu{z})(w) K_\nu(z,w) \bmeas{w} \, .\nonumber
\end{align}\
It is well-known that the map $\mathcal L(\Berg) \ni S \mapsto \Berz(S)$ is injective  (see \cite[Proposition 3.1]{Z12} for a similar proof).
The {\it $\nu$-Berezin transform} of a measurable function $f:\nball \to \C$   is defined by
$$B_\nu (f)(z)=\ip{f\kernu{z}}{\kernu{z}}_\nu$$
whenever the right hand site exists.
If $f: \nball \to \C$ is a measurable function then the the {\it Toeplitz operator with symbol $f$} is defined as the
operator $\Toep_f :\Berg \to \Berg$ given by
\begin{equation}\label{def:Tf}
\Toep_f(g)(z) = P^{(\nu)} (fg ) (z)= \int f (w) g(w) K_{\nu}(z ,w)\bmeas{w},\quad z\in \B^n,  g\in \Berg,
\end{equation}
if the integral exists. If $f$ is bounded then $T_f$ exists, is continuous, and $\|\Toep_f\|_\nu \le \|f\|_\infty$. Furthermore,
the map $L^\infty (\nball) \to \bdd$, $f\mapsto \Toep_f$, is injective (see \cite{E92}).   

Then it follows that
\begin{align}
B^{(\nu)} (\Toep_f)(z&)= \ipnu{f\kernu{z}}{\kernu{z}}
=B_\nu (f)(z) 
\\
&= (1-|z|^2)^\nu \int \frac{f(w)}{|1-\ip{z}{w}|^{2\nu}}\bmeas{w}\nonumber \\
&=   c_\nu \int f(w) \frac{(1-|z|^2)^\nu (1-|w|^2)^{\nu-(n+1)}}{|1-\ip{z}{w}|^{2\nu}} \leb{w}.\label{def:Bf} 
\end{align}
  
\begin{theorem}\label{thm:Binj} Let $\nu >n+1$ then the map $\mathcal L(\Berg) \ni S\mapsto B^{(\nu)}(S)$ is injective. In particular,
$B_\nu : L^\infty (\nball ) \to C (\nball)$ is injective.
\end{theorem}
\begin{proof}   The proof is the same as the proof of \cite[Prop. 3.1]{Z12}.\end{proof}  
  
\subsection{Finite rank operators}
In this subsection we start with an arbitrary Hilbert space $\cH$ and then specialize to the
case $\cH = \Berg$.
As before $\cL(\cH) $ denotes the space of bounded operators on $\cH $. Denote by $\cT^p(\cH)$ the Schatten $p$-class of
operators. Then $ \cT^1(\cH) $, the space of {\it trace class operators} and $\cT^2(\cH)$ is the Hilbert space  $\HS (\cH)$ 
of Hilbert-Schmidt operators with inner product
\[\ip{A}{B} = \tr B^*A = \tr AB^* .\]
 Given $u,v\in \cH $, denote, as before, the rank one operator onto $\C u$ by  
 \begin{equation}\label{def:Suv}
 u\otimes v = S_{u,v}: x\mapsto \ip{x}{v}u .
 \end{equation}  
Then 
\[ S_{u,v} S_{w,z} = \ip{w}{v} S_{u,z}, \quad S_{u,v}^*=S_{v,u} \quad \text{ and } \tr(S_{u,v})=\ip{u}{v}.\]
In particular
\begin{equation}\label{eq:ipTens}
\ip{S_{u,v}}{S_{w,z}}=\tr S_{u,v} S_{w,z}^* =  \tr S_{u,v} S_{z,w} =  \ip{u}{w}\ip {z}{v} = \ip{u\otimes w}{z\otimes v}.
\end{equation}
 Denote by $\bcH$ the conjugate Hilbert space, ie., the space $\cH$ with complex multiplication $\lambda \cdot u =\bar \lambda u$
 and inner product $\ip{u}{v}^-= \ip{v}{u}$. Then \eqref{eq:ipTens} shows that the map $\cH \otimes \bcH \to \rB(\cH)$ extends
 to an isometry inclusion $\cH\otimes \bcH\hookrightarrow \HS (\cH)$. This map is in fact an isometric isomorphism.
 
Let now $\cH = \Berg$ and take $u = f_\vm$ and $v=f_\bn$. Then, using that $K(z,w) =\sum_\bk f_{\bk} (z) \overline{f_\bk (w)}$, 
(see Thm. \ref{lem:ProAa}),
we get for $ S_{f_\bm,f_\bn} =: S_{\bm,\bn}$:
\begin{equation}\label{SKw}
S_{\bm,\bn}K_w =   \ip{K_w}{f_\bn}f_\bm = 
\sum_\bk \ip{f_\bk}{f_\bn}\overline{f_\bk (w)}f_\bm 
=\overline{f_\bn(w)}f_\bm.  
\end{equation}
In particular, if $\bm=\bn$, then 
\[SK_w(z) = f_\bm (z)\overline{f_\bm(w)}.\]
If $\rP^m(\C)\subset \Berg$ is  the space of homogeneous polynomials of degree $m$ then the orthogonal projection is given by
$\prma = \sum_{\bm = m} f_\bm \otimes \overline{f_\bm} : \Berg \to \rP^m(\C)$ 
and the above shows that 
\[\prma (K_w)(z)  = \sum_{|\bm | =m} f_\bm(z)\overline{f_\bm (w)}.\] 

We now use \eqref{SKw} to determine $B_\alpha (S)$ with $S=S_{f_\bm,f_\bn} $:
\begin{align*}
B^{(\nu)} (S)(z)&= (1-|z|^2)^\nu \ip{SK_z}{K_z}_\nu\\
&=(1-|z|^2)^\nu \overline{f_\bn (z)}\ip{f_\bm }{K_z}_\nu\\
&=(1-|z|^2)^\nu \overline{f_\bn (z)}f_\bm(z) 
\end{align*}

This now implies the following Lemma
\begin{lemma}\label{lem:BS}
 Let $\vm,\bn \in \N_0^n$ and $k\in \N_0$. Then the following holds:
 \begin{enumerate}
 \item $B^{(\nu)}  (f_\bm \otimes \overline{f_\bn}   ) (z) = 
 (1-|z|^2)^\nu \overline{f_\bn (z)}f_\bm (z)$. 
 \item  $B^{(\nu)} (f_\bm  \otimes \overline{f_\bm} ) = (1-|z|^2)^\nu |f_\bm (z)|^2$. In
 particular
$B^{(\nu)}  (1\otimes \bar 1)(z) = (1-|z|^2)^\nu$.
 \item $B^{(\nu)}  (\prma )(z) = (1-|z|^2)^\nu \sum_{|\bm | = m} | f_\bm (z)|^2=(1-|z|^2)^\nu \frac{\Gamma (\nu + m)}{m!\Gamma (\nu )}|z|^{2m}>0$.
 \end{enumerate}
  \end{lemma}

 \subsection{The group $\mathrm{SU}(n,1)$}
 The group $\SU (n,1)$ plays an important role in understanding functions on $\B^n$ as well as in the formulation of quantum harmonic analysis (QHA) on the Bergman space $\Berg$. We therefore give a short discussion of the properties of $\SUn :=\mathrm{SU}(n,1)$ that we will need.
 
 The inner product between vectors  $z,w\in\C^n$ is written as
 $\ip{z}{w} = \sum z_j \overline{w_j}$ and the $\C$ bilinear form is  $z\cdot w$, or short $z\cdot w = zw$ is
 given by $zw = \sum_{j=1}^n z_j w_j$, Hence
 $\ip{z}{w} = z\overline{w}$. The group $\SU (n,1)$ is  the group of $(n+1)\times (n+1)$-matrices that leave the form 
  \[\beta_n(z,w) =\ip{Jz}{w}= -z_1\bar w_1 - \ldots -
 z_n\bar w_n + z_{n+1}\bar w_{n+1}, \quad J_{n,1} = \begin{pmatrix}
    -I_n & 0 \\
    0 & 1
  \end{pmatrix}\]
  invariant. Writing $(n+1)\times (n+1)$-matrix  $g$   in block form as 
\begin{equation*}
x(A,w,\xi,d)=
 \begin{pmatrix}
    A & w^t \\
   v & d
  \end{pmatrix},\quad A\in \rM_n(\C),\,\, w,v \in\C^n,\text{ and } d\in\C .
\end{equation*} 
We have $x\in \SUn$ if and only if  $\det x =1$ and
$x^*J_{n,1}x=J_{n,1}$ or $x^{-1} =J_{n,1}x^*J_{n,1}$. Similarly $X$ belongs to the Lie algebra $\sun$ of $\SUn$ if and only if $X= -J_{n,1} X^* J_{n,1}$.
We write $\theta (x) = (x^*)^{-1}  $ respectively $\theta (X) = -X^*$, $x\in\SUn$ respectively $X\in \sun$. The homomorphism
$\theta$ is a {\it Cartan involution}, the standard Cartan involution on $\SUn$.

The group $\SUn$ acts transitively on $\nball$ via the fractional linear transformations
\begin{equation*}
x (A,w,v,d) \cdot z=  \frac{Az+w}{v\cdot z + d},
\end{equation*}
In particular 
$x\cdot 0 = \frac{1}{d}\, w $. This implies  that the stabilizer of $0$, $K=\SUn^0=\{x\in G\mid x\cdot 0 =0\}$ is the group

\begin{equation}\label{K}
  K = \left\{ \left.  x_k=\begin{pmatrix}
      k & 0 \\
      0 & \overline{\det(k)}
    \end{pmatrix} \,\right|\,  k \in \rU(n)\right\}\simeq \rU(n) = \SUn^\theta =\{x\in \SU(n,1)\mid \theta (x) = x\}.
  \end{equation}
  The map $\SUn/K \to \nball$,  $xK \mapsto x\cdot 0$, defines a $\SUn$-isomorphism $\SUn / K \simeq \nball$, identifyng the unit ball $\B^n$ as a homogeneous space.

  \begin{remark} We note here that $Z$, the center of $\SUn$, is contained in $K$ and acts trivial on $\nball$. Hence $\mathrm{PSU}(n,1)=
  \SUn/Z$
  acts on $\nball$ by $aZ\cdot z= az$. This is well defined
  because the center acts trivially because $Z\subset K$ and $z\cdot (x\cdot 0) = x (z\cdot 0) = x\cdot 0$. Furthermore if $G_1\to \SUn$ is
  a covering with covering map $\kappa : G_1\to \SUn$, then $G_1$ acts on $\nball $ by $x\cdot z = \kappa (x) \cdot 0$.
  \end{remark}

  Define the functions $j, j_\nu : G\times \nball \to \C$ by 
  \begin{equation}\label{def:jgz}
  j(x ,z) =  v\cdot z + d \quad\text{and}\quad 
  j_\nu (x , z) = ( v \cdot z + d)^{-\nu}, \quad x=x(A,w,v,d)\in \SUn.
  \end{equation} 
In particular $j(x,z) = j_{-1}(x,z)$.
The function $j_\nu$ is only defined on $\SUn\times \nball$ if $\nu \in \N_0$.  For general $\nu$ one needs to
replace $\SUn$ by a covering group. The covering can be taken to be finite 
  if $\alpha\in \Q$ but otherwise one has to use the simply connected covering group $\wSUn$.  Another way
  around this problem  is to use a central extension as discussed in  \cite{A06,C96,CDO19,M58}.   The following
  is well know:

  \begin{lemma}\label{lem:Kgzw} The following holds for all $\nu>n$   $z,w\in \nball$ and $x,y\in \wSUn$.
  \begin{enumerate}
  \item  $j_\nu(xy ,z) = j_{\nu}(x,y\cdot z)j_\nu (y,z)$.
  \item  $j_\nu (x^{-1},gz) = j_{-\nu} (g,z)$.
  \item $j(x,z) \overline{j(x,w)}   h(x\cdot  z,x\cdot w) =h(z,w)$.
  \item $j_\nu (x,z)\overline{j_\nu (x,w)} K_\nu (x\cdot z,x\cdot w) = K_\nu (z,w)$.
  \end{enumerate}
  \end{lemma}
  \begin{proof}  (1) and (3) are direct calculations.  (2) follows from (1) as $j_\nu (\rI ,z) = 1$ for all $z$. (4) follows from (3).
  \end{proof}

  \subsection{Polar coordinates and Integral formulas}
  We also have that $K$ acts transitively on $\rS^{2n-1}$ and the stabilizer of $e_1$ is
  \[M = \left\{ \left. m_k = \begin{pmatrix} 1 & 0 & 0\\ 0 & k & 0\\ 0 & 0 & \overline{\det k}\end{pmatrix}\, \right| \, k\in \rU (n-1)\right\}
  \simeq \rU (n-1) .\]
  Thus $\rS^{2n-1} \simeq K/M$.
  We also define the group
  \[A = \left\{\left. a_t= \begin{pmatrix}  \cosh (t) & 0 & \sinh (t)\\ 0 &\rI_{n-1} & 0\\ \sinh (t) & 0 & \cosh (t)  \end{pmatrix}\, \right| \, t\in \R\right\}.\]
  Note that $a_t\cdot 0 = \tanh (t) e_1$ and that the group $M$ commutes with $A$ and
  its Lie algebra $\fa = \R h\simeq \R$, where $h= E_{1,n+1} + E_{n+1,1}$ is the
  operator that interchanges $e_1$ and $e_{n+1}$ and is zero on the orthogonal complement of
  $\C e_1 \oplus \C e_{n+1}$.
   Note that $a_t\cdot 0 = \tanh (t) e_1$. Thus if $0\not=z = r\omega \in \nball$ then
  $z= ka_{\tan^{-1}(|z|)}\cdot 0$ for some $k\in K$. It follows, with  $A^+ =\{a_t\in A\mid t>0\}\simeq \R^+ = (0,\infty)$, that we have
   \[ \SUn = KAK,\quad K A\cdot 0 = \nball\quad\text{and}\quad  KA^+\cdot 0=\nball \setminus \{0\}.\]
This implies that we have the following polar coordinates for the ball:
  \begin{equation}\label{eq:polarCo}
  \nball \setminus \{0\} \simeq K/M \times (0,\infty) =\rS^{2n-1}\times (0,\infty)\quad\text{where } (\omega , t) \mapsto 
   \tanh (t)\omega.
   \end{equation}

  Note that if
 $p : \SUn\to \SUn/K \simeq \nball$ is the natural projection $p (a)= a\cdot 0$, then
 the map $f\mapsto f\circ p$ defines a linear bijection between functions on the ball and
 {\it  right} $K$-invariant functions on the group.

  Let $w = \diag (i,1,\ldots , 1,-i)$ where $\diag(a_1,\ldots,a_{n+1})$ stands for the
  diagonal matrix with diagonal elements $a_{ii} = a_i$. 
  \begin{lemma} $whw^{-1} = - h$ and $wa_tw^{-1} = a_{-t} = a_t^{-1}$.
  \end{lemma}
  \begin{proof} As $a_t = \exp th$ and $a_{-t} =a_t^{-1}$ it is enough to show that $whw^{-1}= -h$, but that follows
  from direct matrix multiplication.
  \end{proof}
  \begin{lemma}\label{lem:fKinv} Let $f: \SUn \to \C$ be $K$-biinvariant. Then $f(x) = f(x^{-1})$ for all
  $x\in G$.
  \end{lemma}
  \begin{proof} We write $x =ka_th$ with $k,h\in K$. If $f$ is $K$-biinvariant we have
  \[f(x^{-1}) = f(h^{-1}a_{-t} k^{-1}) = f (a_{-t} ) = f(w a_t w^{-1}) = f(a_t) =f(x).\qedhere\]
    \end{proof} 
  \begin{lemma}\label{lem:Jac}  Let $\varphi : \nball \to \nball$ be a biholomorphic
  diffeomorphism, $w = \varphi^{-1}(0)$ and let $J_\varphi (z) $ the holomorphic Jacobian of $D\varphi (z)$. Then the following
  holds:
  \begin{enumerate}
  \item The real Jacobian of the action of $\varphi$ is given by $J_\R \varphi (z) = |J_\varphi (z)|^2$.
  \item $\ds J_\R \varphi (z) = \left(\frac{h(w)}{|h(z,w)|^2}\right)^{n+1}=
  \left(\frac{1-|w|^2}{|1-\ip{z}{w}|^2}\right)^{n+1}$. 
  \end{enumerate}
  \end{lemma}
  \begin{proof} Both statements can be found  in \cite[p. 7 and 8]{Z05}.  
  \end{proof}
  
  Let us reformulate this so it fits our purposes. Thus take $\varphi (z) = x\cdot z$, $x\in \SUn$. Then
  $w = x^{-1}\cdot 0$. Using that $j(x,z)\overline{j(x,w)} h(x\cdot z,x\cdot w)= h(z,w)$ we get
  \[1- |w|^2 = h(x^{-1}\cdot 0, x^{-1}\cdot 0) = |j(x^{-1},0)|^{-2}\]
  and
  \begin{align*}
  | h(z,w)|^2 &= |h(x^{-1}\cdot (x\cdot z) , x^{-1}\cdot 0)|^2 = |j(x^{-1},x\cdot z)|^{-2} | j(x^{-1},0)|^{-2}\\
  & = |j(x,z)|^2 | j(x^{-1},0)|^{-2}.
  \end{align*}
  Thus 
  \begin{equation}\label{eq:z05}
   \left( \frac{h(w)}{|h(z,w)|^2}\right)^{n+1} = |j(x , z)| ^{-2(n+1)} \quad\text{or}\quad \frac{h(w)}{|h(z,w)|^2}= |j(x,z ) |^{-2}
  \end{equation}
  
  Using that $|j(y,0)|= h(w)$ this also leads to the following useful lemma:
   
 \begin{lemma}\label{lem:hxy}
 Let $x,y\in \SUn$ and write $z=x\cdot 0$ and $w=y\cdot 0$. Then
 \[h(y^{-1}x) = \frac{h(z)h(w)}{|h(z,w)|^2} .\]
 \end{lemma} 
 
 We now use \eqref{eq:z05} and Lemma \ref{lem:Jac} to prove the following:

  \begin{theorem}\label{thm:Int} The following holds true:
   \begin{enumerate}
   \item If $f \in L^1(\nball , dz)$, then
  \[\int_{\nball} f (x \cdot z)|j(x,z)|^{-2(n+1)} \leb{z} = \int_{\nball} f(z) \leb{z}, \quad x\in \SUn .\]
  \item  Assume that $\nu \in\R $ and that $f\in L^1(\nball , h(z)^{\nu - n -1}\leb{z})$. Then
  \[\int_{\nball} f(x\cdot z) |j  (x ,z)|^{-2\nu} h(z)^{\nu - n -1} \leb{z} = \int_{\nball } f(z) h(z)^{\nu - n -1} \leb {z},\quad x\in\SUn .\]
  \item The $\SUn$-invariant measure on $\nball $ is, up to a positive constant, given by 
  \[\inv{z}=\frac{\leb{z}}{(1-|z|^2)^{n+1}}.\]
 \item Let $f\in L^1(\nball ,\lambda_{\nball})$.
 Then there exists 
constants $c_1,c_2>0$ such that  
\begin{align*}
  \int_{\nball} f(z) \inv{z}& = c_1 \, \int_K \int_0^\infty \!\!\!\! \int_K f(k a_t\cdot 0) 
  \sinh(2t)\sinh(t)^{2(n-1)} \,\haarK{k} \,dt 
  \\
  &= c_2\int_{\SUn} f(x\cdot 0)\haar{x}\nonumber\\
  &= c_2 \int_{\SUn} f\circ p (x) \haar{x}.
  \end{align*}
  \end{enumerate}
  \end{theorem}
  \begin{proof} By Lemma \ref{lem:Jac}, \eqref{eq:z05} and the change of variables formula for $\R^{2n}$ we have for $f\in L^1(\nball, dz)$:
  \[\int_{\nball} f(x\cdot z) |j(x,z)|^{-2(n+1)} dz = \int_{\nball} f(z) dz \]
  which is (1).
  
 (2)  Let $f\in L^1(\nball , h(z)^{\nu - n -1}\leb{z})$ and $x \in \SUn$. Then we get by Lemma \ref{lem:Kgzw} 
\[
  |j(x,z) |^{-2\nu}h(z)^{\nu - n -1}  = |j(x,z)|^{-2(n+1)} |j(x,z)|^{-2\nu + n +1} h(z)^{\nu - n -1} = |j(x,z)|^{-2n} h(x\cdot z)^{\nu -n -1}\]
  and hence
  \begin{align*}
  \int_\nball f(x\cdot z) |j(x,z)|^{-2\nu} h(z)^{\nu -n-1}dz &= \int_\nball f(x\cdot z) |j(x,z)|^{-2 (n+1)} h(x\cdot z)^{\nu -n-1}dz\\
  &= \int_{\nball} f(z) h(z)^{\nu -n-1}dz .
  \end{align*}
  
  (3) This follows from (2) by taking $\nu = 0$.
  
  (4) This is \cite[Thm. 5.8]{Helgason84}. 
  \end{proof}

    Now, we normalize the Haar measure s.t. $c_2=1$. Then the following holds:
    \begin{align}\label{eq:integral}
        \int_G f\circ p (x) \haar{x}= \int_{\B^n} f(z) \inv{z}.
    \end{align}
   
    We point out the following consequence of Theorem \ref{thm:Int}. 
    \begin{theorem} \label{thm:pi}
    For $f\in L^2(\nball , \mu_\nu)$ define
    \begin{align}\label{eq:pinu}
        \widetilde{\pi}_\nu (x) f (z) = j_{\nu}(x^{-1},z)f (x^{-1}\cdot z) ,\quad z\in \nball .
    \end{align}
    Then $\widetilde{\pi}_\nu$ is a unitary representation of $G$ (or covering)  acting  on $L^2(\nball,\mu_\nu)$ and
    $\Berg$ is an invariant and irreducible subspace. The corresponding representation is also denoted by $\Tilde{\pi}_\nu$.
    \end{theorem}
    \begin{proof} Everything except the irreducibility of $\Berg$ follows from Theorem \ref{thm:Int}, part 2. That $(\Berg,\Tilde\pi_\nu)$ is
    irreducible is well known.
    \end{proof}

    \subsection{Automorphisms of $\B^n$}
Let $\Aut(\nball )$ denote the group of automorphisms of $\nball$, i.e., the group of biholomorphic diffeomorphism of $\nball$. Denote by $\Aut (\nball )_e$
 connected component of $\Aut (\nball)$ containing the identity map. Then it follows from \cite[Lem. 4.3 {\&} Thm 6.1]{H01} that
 $\Aut (\nball )_e = \SUn/Z$ where $Z$ denotes the center of $\SUn$. Note that $Z\subset K$. Then for each $z\in \nball $, there exists an element $\tau_z\in \Aut (\nball)$ 
    of period two that interchanges $z$ and $0$. 
  We note that $\tau_z$ is in general not an element of $G$ as $G$ is a finite or infinite covering
    of $\Aut (\nball)$. But there always exists a $c_z\in Z(G)$, the center of $G$ and $t_z\in G$ such that $t_z(z) =0$, $t_z(0) =z$ and $t_z^2 = c_z$.
    The following lemma is a crucial one.
   \begin{lemma}\label{lem:Aut}
    Let $x\in G$. Then $x=t_{x\cdot 0} k_x$ for some $k_x\in K$. In particular, if $G=\Aut (\nball )$ then $x = t_{x\cdot 0} k_x$.
\end{lemma}
\begin{proof} Let $z= x\cdot 0$. Then   $t_x^{-1} x\cdot 0 = 0$ and hence $t_x^{-1}x = k_x\in K$. 
\end{proof} 


 \section{The commutative Banach algebra $\LoneinvK$}\label{sec:L1K}
 
 The main purpose of this section is to discuss the algebra $\LoneinvK$ of all radial functions on $\B^n$ that are integrable with respect to the invariant measure. We begin by defining the convolution product. For that, recall that if $(\pi,\cB)$ is a Banach space representation,
then we can define the convolution of compactly supported continuous functions $f\in C_c(G)$ with elements of $\cB$, also
called the integrated representation, by
\[f*_{ \pi} v = \pi (f) v = \int_\SUn f(x)  \pi (x) v\, d\mu_{\SUn}(x), \quad v\in \cB\]
where the integral is taken in the weak sense. If $\pi$ is norm-preserving or more generally bounded, the convolution
extends to all of $L^1(\SUn)$.  This in particular with the left   translation of a function $f:G\to \C$, respectively $f:\nball \to \C$, given
by 
$$\actfL{x}{f}(y)=f(x^{-1}y),\quad x\in G, \, y\in \SUn \text{ respectively } y\in \nball $$  
which leads to the classical convolution  
\[f*g(x) =\int_{G} f(y) \actfL{y}{g}(x)\haar{y}= \int_{G} f(y) g(y^{-1}x)\haar{y}. \]
This is particular exists if $f\in L^1(G)$ and $g$ is a $L^p$-function on $\SUn$ or  on the ball.
Finally, by this we also have a convolution of $L^1$-functions and elements in $L^1(\SUn,\mu_\nu)$ respectively the $\Berg$ given by
\[f*_{  \nu } g (z) = \widetilde{\pi}_\nu (f) g(z) = \int_{\SUn} j_\nu (x ^{-1},z ) f(x)g(x^{-1}z) d\mu_{\SUn}(x),\]
where, $\pinu(f)$ is the integrated representation associated with $\pinu$ defined in \eqref{eq:pinu}.

Recall the function $\hnu$ defined in \eqref{eq:hnu}. 
For $\nu > n$ we let
\begin{equation}\label{def:varphin}
\varphi_\nu (z) = c_\nu \hnu { (z)}= c_\nu h(z)^\nu.
\end{equation}
Then $\int_G \varphi_\nu(z)\inv{z}=1$.
 \begin{lemma}\label{lem:Berezin_function}
     Let $1\leq p\leq\infty$ and let $f\in \Lpinv$. Then $B_\nu(f) = f\ast \varphi_\nu$. 
 \end{lemma}
 \begin{proof} This follows from \eqref{def:Bf}, \eqref{eq:z05}  and Lemma \ref{lem:hxy}.  
\end{proof}

We also have the following
\begin{lemma}\label{lem:phiCophi} Let $\nu > n$, then there exists a positive constant  $d_\nu $ such that 
\[B_\nu (\varphi_\nu) (z) = \varphi_\nu \ast \varphi_\nu (z)  = d_\nu \varphi_\nu (z) >0 \]
\end{lemma}
\begin{proof} We have, by identifying $\varphi_\nu $ with $\varphi_\nu \circ p$ and noting that
$\varphi_\nu (x)  = c_\nu \ip{\tilde{\pi}_\nu (x)1}{1}_\nu$ 
\begin{align*}
\varphi_\nu \ast \varphi_\nu (x\cdot 0) & = c_\nu^2 \int_G \ip{\tilde{\pi}_\nu (y)1}{1}_\nu \ip{\tilde{\pi}_\nu (y^{-1}x)1}{1}_\nu d\mu_G(y)\\
& = c_\nu^2 \int_G  \ip{\tilde{\pi}_\nu (y)1}{1}\ip{\tilde{\pi}_\nu (x) 1}{\tilde{\pi}_\nu (y)1}_\nu d\mu_G(y)\\
& = c_\nu^2 \frac{1}{d(\tilde{\pi}_\nu )} \ip{\tilde{\pi}_\nu (x)1}{1}_\nu\\
&= \frac{c_\nu}{d (\tilde{\pi}_\nu)} \varphi_\nu (x)
\end{align*}
where we have used the orthogonality relation for the square-integrable representation $\tilde{\pi}_\nu$ and $d (\tilde{\pi}_\nu)$ is
the formal dimension.
\end{proof}

Recall that $K=\rU(n)$ can be identified as a subgroup of $G$. We say a function $f:\B^n\to \C$ is {\it radial} or {\it left $K$-invariant}, if
$$\actfL{k}{f}=f,\quad \forall k\in K.$$
The radialization $\rad f$ of a function $f:\B^n\to \C$ is given by
$$f^K(z) = (\rad f)(z)=\int_{K} (\actfL{k}{f})(z)\haarK{k}= \int_L f(kz) \haarK{k}, \quad z\in \B^n $$
where $\mu_K$ is the normalized Haar measure on $K$.

 We remark that if $f$ and $g$ are functions on the ball then $f*g$ is not well defined, however if $g$ is $K$-invariant then, viewing $f$
as a right $K$-invariant function on the group the convolution $f*g$ is well defined and we have
\begin{lemma} Let $f,g \in L^1(\nball )$ then the following holds:
\begin{enumerate}
\item $\ds f*g (z) = f*g^K(z)$.
\item If $f$ is left $K$-invariant then $f*g$ is left $K$-invariant.
\end{enumerate}
\end{lemma}
\begin{proof} This is a simple calculation, see also \cite{DDMO25}.
\end{proof}
 
 We also have the following:
\begin{lemma}\label{lem:radial_functions}
    Let $1\leq p\leq \infty$ and let $f\in \Lpinv$. Then the following are equivalent:
    \begin{enumerate}
 \item $f$ is radial
        \item $\rad f=f$
        \item  There exists $\nu >n+1$ such that  $B_\nu (f)$ is radial.
        \item  $B_\nu (f)$ is radial for all $\nu>n+1$. 
 \item  $f\circ p (x)=f\circ p (x^{-1})$ for all $x\in G$, where $p :G\to G/K$ is the canonical projection.
 \end{enumerate}
\end{lemma}
\begin{proof}  The equivalence of (1) and (5) is a consequence of the assumption that $f$ is right invariant and Lemma \ref{lem:fKinv}.
(5) $\Rightarrow (4)$ As the inner product and norm is invariant under the group $K$ it
follows that $B^{(\nu )} (f\circ \ell_k) = B^{(\nu )} (f) \circ \ell_k$, $k\in K$. Hence
the $K$-invariance of $f$ implies that $B^{(\nu)}(f)$ is $K$-invariant.  
(4) $\Rightarrow $ (3) is clear. (3) $\Rightarrow (2)$ follows from the injectivity of the  $\nu$-Berezin transform, see
Lemma \ref{thm:Binj}.
Finally (2) $\Rightarrow $ (1) follows as $\rad f $ is $K$ left-invariant.  
\end{proof}
 
Following the notation in Appendix \ref{app:Gelfandpairs}, given a space $\cF$ of functions on the ball, we denote by 
$\cF^K$, the set of all $K$-invariant functions in $\cF$. From Appendix \ref{app:Gelfandpairs}, we 
have that $\LoneinvK$, the space of all $K$-invariant functions on $\nball$ that are integrable with
 respect to the invariant measure $\lambda$, forms a commutative Banach algebra with the convolution product. 
Moreover, $\LoneinvK$ is a commutative Banach $*$-algebra, with the involution given by 
$f^*(x) =\overline{f (x^{-1})}$. Note that in our case we have $f(x^{-1}) = f(x)$ so
$f^* = \bar{f}$. We have $\|f^*\|_1 = \|f\|_1$ and $\|f^* * f\|_1 \le \|f\|_1^2$. In other words, $(G,K)$ is a 
Gelfand pair and the unit ball $\nball $ is a commutative space. We refer to the Appendix \ref{app:Gelfandpairs} for general discussion
about commutative spaces, the spherical functions and the Gelfand transform.  

\subsection{Spherical functions}\label{sec:spherical_functions}
 The space of bounded homomorphism  of $\LoneinvK$ is given by the set of all bounded spherical functions $\spheri=\spheri(\B^n)$, as in \ref{subsec:sphericalFT}. The space of $*$-homomorphism is given by the
 space  $\spheripd$ of positive definite spherical functions.   

Recall the Gauss hypergeometric function  
\[{}_2F_1 (a, b;c;z) = \sum_{k=0}^\infty \frac{(a)_k(b)_k}{(c)_k} \frac{z^k}{k!}, \quad c \not\in -\N_0, \, \, |z|<1 \]
 where for a complex number $z\in \C\setminus (-\N)$.
\[(z )_k = z (z +1) \cdots (z +k-1) = \frac{\Gamma (z +k)}{\Gamma (k)} .\]
For values $z \in \mathbb C \setminus \{1\}$ the hypergeometric functions are then defined through analytic continuation. We refer to \cite[Chap. 9]{L73} and \cite[Chap. 2]{AAR} for information about the hypergeometric functions.
The function $\hgf (a,b; c; z) $ is the unique solution to the differential equation 
\[z(z-1)f^{\prime\prime} (z) + [(a+b+1)z - c]f^\prime (z) + ab f(z) = 0\]
that is regular at zero and normalized by $f(0) =1$.
 We have the following:
 
\begin{theorem} \label{thm:dform}
Let $a,b,c\in \C$, $c\not\in -\N_0$. Then the following holds:
\begin{enumerate} 
\item $\ds z\mapsto \hgf (a,b; a+b+1/2; z )$ extends to a holomorphic function on $\C \setminus [1,\infty)$.
\item $\ds  \hgf (a,b; a+b+1/2; z) = \hgf \left(2a, 2b; a+ b+ 1/2; \frac{1-\sqrt{1-z}}{2}\right)$ .\label{eq:hfgdoubl}
\item $\ds \hgf ( a,b;c;z) =  (1-z)^{-a} \hgf \left(a, c-b; c; \frac{-z}{1- z}\right)$, $z\in\C\setminus [1,\infty)$.
\end{enumerate} 
\end{theorem}
\begin{proof}  For the doubling formula (2) see  \cite[\S 9.6,p.309 ]{L73}. For (3) see \cite[p. 305]{L73}.
\end{proof}

As before we let $a_t =\begin{pmatrix} \cosh (t) & 0 & \sinh (t) \\ 0 & \rI_{n-1} & 0\\
\sinh (t) & 0 & \cosh(t)\end{pmatrix}$. Then $a_t\cdot 0 = \tanh (t)e_{n+1}$.
and, as $KA\cdot 0$, it follows that the spherical functions can be viewed as a function of $t$ only. This also follows from the fact that $\phi$ is radial. We write $\phi (t)$
for the spherical functions $\phi $ evaluated at $\tanh (t) e_{n+1}$.
  
Recall that any $K$-invariant function on $\nball$ can be written as a function of $|z|$, i.e., there exists 
a function $g_f: \R^+ \to \C$ such that $f(z) = g_f(|z|)$. By abuse of notation we simply
write $f(z) = f(|z|)$.

\begin{theorem}[Spherical functions]\label{thm:spherical_functions} Let $\rho = n -1/2$. The spherical functions on $\nball$ are parametrized by $\lambda\in\C$ and
given by the hypergeometric
functions in the following way:
\begin{enumerate} 
\item $\ds \phi_\lambda (t) = \hgf \left(\frac{\rho +i\lambda}{2} , \frac{ \rho - i\lambda}{2}   ;n; -\sinh^2(t)\right)
= \hgf \left( \rho+i\lambda , \rho-i\lambda; n; \frac{1-\cosh (t)}{2}\right) $.
In particular, if $\lambda\in\R \cup i\R$ then $\phi_\lambda$ is real valued.
\item We have
\begin{align*}
  \phi_\lambda (z) = \phi_\lambda (|z|) & = \hgf \left(\frac{\rho +i\lambda}{2} , \frac{ \rho - i\lambda}{2}   ;n; 
\frac{-|z|^2}{1-|z|^2 }\right)\\
&= (1-|z|^2)^{\frac{1}{2}(\rho+i\lambda)}\hgf \left(\frac{\rho + i\lambda}{2}, \frac{\rho+1+i\lambda}{2}; n ; |z|^2\right) .
\end{align*}
\item  $\phi_{-\lambda}(t) = \phi_\lambda (t) = \phi_\lambda (-t)$.
\item $\phi_\lambda = \phi_\mu $ if and only if $\mu = \pm \lambda$.
\item The spherical functions are positive definite if and only if
\[\lambda \in \R \cup i [-n,n].\]
\item The spherical function $\phi_\lambda$ is bounded if and only if $\im \lambda \in [-n,n]$.
\item 
\[d\mu_\fa (\lambda ) = c\frac{ | \Gamma (i\lambda)|^2}{|\Gamma (\frac{1}{2}(\rho + i\lambda))\Gamma (\frac{1}{2}(n-1 + i\lambda )|^2}\, d\lambda\]
Then one can choose $c>0$ such that the spherical Fourier transform extends to an unitary isomorphism
\[L^2(\nball , \mu_{\nball})^K \simeq L^2(\R^+,\mu_\fa )\]
and $\ds \widehat {C_c^\infty (\nball )^K}$ is dense in $L^2 (\R^+,\mu_\fa)$.
\item If $f\in C_c^\infty (\nball )^K$ then
\[f(z) = \int_{\R^+} \widehat{f}(\lambda ) \phi_\lambda (z) d\mu (\lambda) .\]
\item When $\lambda_j \to \lambda$, then $\phi_{\lambda_j} \to \phi_\lambda$ pointwise.
\end{enumerate}
\end{theorem}
\begin{proof} (1) The first part is \cite[p. 484]{Helgason84} and the second
identity is the doubling formula from  Theorem \ref{thm:dform}.  

(2) follows from (1) using that with $z=a_t\cdot 0 =\tanh (t)$ we have  
\[1- |z|^2 =1- \tanh^2 (t)= \frac{1}{\cosh^2(t)} =\frac{1}{1+\sinh^2(t)}\]
and hence $1+\sinh^2(t) =\frac{1}{1-|z|^2}$ or 
\[-\sinh^2(t) = \frac{-|z|^2}{1-|z|^2} = \frac{|z|^2}{|z|^2-1} .\] 

This follows directly from (1) as $\hgf (a,b; c;z) = \hgf (b,a; c;z)$.

(3)  and (4) follows from (1).

(5) is Theorem 1 in \cite{F-JK79}.

(6) is Theorem 8.1 in \cite{Helgason84}.

(7) and (8) are Theorem 7.5 \cite{Helgason84}.

(9) is straightforward to verify by any of the many integral formulas for the hypergeometric function.
\end{proof}

The spherical representations can all be realized on the Hilbert space $L^2(S^{2n-1})$ by
\[\rho_\lambda (a) f(w)= |j(a^{-1}, w)|^{ i\lambda - \rho} f(a^{-1} w) .\] 
This representation is unitary for $\lambda \in \R$ and the usual $L^2$-inner product on
$L^2(S^{2n-1})$. It is irreducible for $\lambda \not=0$. To construct an invariant sesquilinear form for the other parameters, one
constructs a densely defined singular integral operator $A_\lambda : L^2(S^{2n-1})\to L^2(S^{2n-1}) $
intertwining $\pi_\lambda $ and $\pi_{-i\bar\lambda}$ such that $A_\lambda 1 = 1$.
Then one defines $\ip{f}{g}_\lambda = \ip{f}{A_\lambda g}_\lambda$. But, as $A_\lambda 1_\lambda = 1_\lambda$, it follows that the formula for the spherical function is still the same:  
\begin{align}
\phi_\lambda (a_t)  &= \int_{\rS^{2n-1}} | \omega_1 \sinh t + \cosh t|^{i\lambda -n+1/2} d\sigma_n (\omega)\nonumber\\
& =
c\int_0^\pi |\cos (\theta ) \tanh (t) + 1|^{i\lambda - n +1/2} \cosh (t)^{i\lambda - n + 1/2} \sin (\theta )^{2(n-1)}d\theta\label{form:Sph} .
\end{align}

\begin{theorem}\label{lem:Bnvarphi} Assume that $\phi_\lambda$ is bounded and $f\in L^1(\nball )^K$. Then
\[f\ast \phi_\lambda (x) = \phi_\lambda \ast f (x) = \widehat{f}(\lambda ) \phi_\lambda(x) .\]
Moreover,
\[ \widehat{\varphi_\nu}(\lambda ) = B_\nu (\phi_\lambda ) (0) \not= 0  .\]
\end{theorem}
\begin{proof}
The first part is well know but we include the proof. By identifying $\varphi_\lambda$ with the corresponding $K$-bi-invariant function on $G$, we have:
\[f \ast \phi_\lambda (y\cdot 0) = \int_G f(x\cdot 0) \phi_\lambda (x^{-1}y) \haar {x}
=\int_G f(x\cdot 0) \phi_\lambda (x^{-1}k y) \haar{x}  .\]
as $f$ is left $K$-invariant. Integrating over $K$, and by using Fubini's theorem  and the fact that $\phi_\lambda$ is a spherical function, implies that
\[f\ast \phi_\lambda (y\cdot 0)= \widehat{f}(\lambda ) \phi_\lambda (x^{-1}) = \widehat{f}(\lambda ) \phi_\lambda (x) . \]

Lemma \ref{lem:Berezin_function} now implies that 
\[ \widehat{\varphi_\nu} (\lambda ) \phi_\lambda (x\cdot 0 ) = \varphi_\lambda \ast \phi_\lambda (x\cdot 0 ) = B_\nu (\phi_\lambda) (x\cdot 0)  .\]
This implies in particular that if $\widehat {\varphi_\nu} (\lambda ) =0$ then $B_\nu (\phi_\lambda ) = 0$ which contradicts the injectivity of
$B_\nu $. Hence $\widehat {\varphi_\nu }(\lambda )\not= 0 $ for all $\lambda \in \rS_{bd} $, and, as $\phi_\lambda (0) = 1$, we finally have
$0\not= \widehat{\varphi_\nu}(\lambda ) = B_\nu(\phi_\lambda ) (0)$, which implies that $ B_\nu(\phi_\lambda ) (0)\not= 0$.
\end{proof}
 
As we have already mentioned before, it is well-known that the continuous homomorphisms from $\LoneinvK$ to $\mathbb C$, i.e., the multiplicative linear functionals of $\LoneinvK$, agree with the bounded spherical functions. Nevertheless, we could not locate a reference where the topology of the Gelfand space $\mathcal M(\LoneinvK)$ is also described in this picture. This is certainly well known, given that the proof we show below is quite straightforward. Nevertheless, we add this for completeness. Recall that the Gelfand space $\mathcal M(\LoneinvK)$ is, by definition, the space of all nontrivial multiplicative linear functionals of $\LoneinvK$ endowed with the weak$^\ast$-topology.
\begin{theorem}\label{thm:gelfand_space_function_algebra}
    As topological spaces, the Gelfand space $\mathcal M(\LoneinvK)$ agrees with $(\mathbb R \times i[-n, n]) / \sim$, where $\sim$ is the equivalence relation of identifying $\lambda$ and $-\lambda$.
\end{theorem}
\begin{proof}
    First of all, we recall how the Gelfand topology on the space of bounded spherical functions is defined: A net $(\phi_\gamma)$ of bounded spherical functions converges to the bounded spherical function $\phi$ if and only if $\langle \phi_\gamma, f \rangle \overset{\gamma}{\rightarrow} \langle \phi, f\rangle$ for all $f \in \LoneinvK$. Since $\LoneinvK$ is separable, the Gelfand topology is metrizable, hence it suffices in the following to consider sequences.

    It was already mentioned in \ref{thm:spherical_functions} that, as sets, the Gelfand space of $\LoneinvK$ agrees with $(\mathbb R \times i[-n, n]) / \sim$. Hence, we only need to prove that the natural topology of $(\mathbb R \times i[-n, n]) / \sim$ (i.e., the quotient topology of the Euclidean topology) agrees with the Gelfand topology. To achieve this, we show that a sequence $(\lambda_j)$ convergence to $\lambda$ in the Gelfand topology if and only if it converges in the natural topology.

    We first assume that $\lambda_j$ is a sequence in $(\mathbb R \times i[-n, n]) / \sim$ converging to $\lambda \in (\mathbb R \times i[-n, n]) / \sim$ with respect to the natural quotient topology of this space. Then, we can lift this to a sequence $\tilde{\lambda}_j$ in $\mathbb R \times i[-n, n]$ converging (in Euclidean topology) to $\tilde{\lambda} \in \mathbb R \times i[-n, n]$. Then, by (9) in \ref{thm:spherical_functions}, $\phi_{\tilde{\lambda}_j} \to \phi_{\tilde{\lambda}}$ pointwise such that (using that all the bounded spherical functions are bounded by $1$, cf.~Lemma \ref{lemma:bddspherical_bound_one}) by the dominated convergence theorem one obtains $\langle \phi_{\tilde{\lambda}_j}, f \rangle \to \langle \phi_{\tilde{\lambda}}, f\rangle$ for each $f \in \LoneinvK$. By point (4) of Theorem \ref{thm:spherical_functions}, it follows that $\lambda_j \to \lambda$ in the Gelfand topology.

    Now, we assume that we have a sequence $(\lambda_j)$ which does not converge to $\lambda$. If the sequence $\lambda_j$ has a subsequence $\lambda_{j_k}$ converging to some $\mu \in (\mathbb R \times i[-n, n]) / \sim$, then we can conclude as before to show that $\phi_{\tilde{\lambda}_{j_k}} \to \phi_{\tilde{\mu}}$ in the Gelfand topology. Since $\phi_{\tilde{\lambda}} \neq \phi_{\tilde{\mu}}$, we see that we also have $\lambda_j \not \to \lambda$ in the Gelfand topology. Hence, we only need to consider the case where the sequence $(\lambda_j)$ has no convergent subsequence. Then, it is clear that (after lifting again to $\mathbb R \times i[-n, n]$) we have $|\tilde{\lambda}_j| \to \infty$. We show that there is no bounded spherical function $\phi$ such that $\langle \phi_{\tilde{\lambda}_j}, f\rangle \to \langle \phi, f\rangle$ for all $f \in \LoneinvK$, which shows that $(\lambda_j)$ does not converge in the Gelfand topology. This will be shown by contradiction, i.e., we assume that such $\phi$ exists.
    
    Let $z$ be such that $|z| < \frac 12$. Then, 
    \begin{align*}
        \phi_{\tilde{\lambda}_j}(z) &= \hgf \left (\frac{\rho + i\lambda}{2}, \frac{\rho - i\lambda}{2}; n; \frac{-|z|^2}{1-|z|^2} \right)\\
        &= \sum_{k=0}^\infty \frac{\Gamma(\rho + k + i\lambda) \Gamma(\rho + k -i\lambda)}{\Gamma(n+k)\Gamma(k)} \frac{(-1)^k|z|^{2k}}{(1-|z|^2)^k k!}.
    \end{align*}
    A simple exercise in integration by parts shows that $\Gamma(\rho + k \pm i\lambda) \to 0$ as $|\lambda|\to \infty$. Hence, dominated convergence shows that $\phi_{\tilde{\lambda}_j}(z) \to 0$ as $j \to \infty$ whenever $|z| < \frac{1}{2}$. Hence, whenever $f \in \LoneinvK$ is supported in the ball with radius $\frac{1}{2}$, we see that $\langle \phi, f\rangle = \lim_{j \to \infty} \langle \phi_{\tilde{\lambda}_k}, f\rangle = 0$. Thus, for $|z| < \frac{1}{2}$ we obtain that $\phi(z) = 0$. Writing $\phi$ using its power series expansion as a hypergeometric function, we see that $\phi(z) = 0$ for all $|z| < \frac 12$ is impossible for a bounded spherical function, as the coefficients in the power series consist of products of Gamma function terms, which never vanish. Hence, we arrive at a contradiction, showing that $\lambda_j$ cannot converge in the Gelfand topology when $|\tilde{\lambda}_j| \to \infty$.    
\end{proof}
 
 Having obtained this explicit description of the Gelfand space as a topological space, it is now an easy exercise in topology to prove the following:
 \begin{corollary}\label{cor:simply_connected}
     $\mathcal M(\LoneinvK)$ is simply connected.
 \end{corollary}
 Based on the previous observations, we will in the following always identify $\spheri$ with the topological space $\mathcal M(\LoneinvK)$ and identify $\spheripd$ with the space of positive definite spherical functions endowed with the subspace topology, or, to phrase things differently:
 \begin{align}
     \spheri \cong (\mathbb R \times i[-n, n]) / \sim, \quad \spheripd \cong (\mathbb R \cup i[-n, n]) / \sim.
 \end{align}
 

\section{Quantum harmonic analysis on the Bergman space}\label{sec:QHA}

In this section, we discuss quantum harmonic analysis (QHA) on the Bergman space, utilizing the representation $\pi_\nu$ of 
the group $G=\mathrm{SU}(n,1)$ acting on the Bergman space. We define the convolution between a function and an operator, and 
the convolutions between two operators. We then place particular emphasis on the radial operators and their QHA 
convolutions  as this leads to a commutative theory and
a Gelfand transform.

For $f$ in $L^2 (\nball ,\mu_\nu) = L^2_\nu$, $x\in \wG$, the universal covering group of $\SUn$ and $z\in \nball$ recall the
unitary representation $\widetilde{\pi}_\nu$ from Theorem \ref{thm:pi}:
\[( \widetilde{\pi}_\nu (x)f) (z) = j_\nu (x^{-1},z)f(x^{-1}z) \]
and its restriction to the Bergman space, which we also denote by $\Tilde{\pi}_\nu$. 

The above representation  is the one mostly used in representation theory and harmonic analysis on Lie
groups. It also serves   well for analysis on the unweighted Bergman spaces (see \cite{DDMO25}). However, in the weighted case,  we
prefer  to work with the corresponding projective representation (see for example \cite{CDO19})  which in particular allows us
to avoid the covering of $\SUn$. This idea was used in \cite{CDO19}, but here we prefer a different realization of the projective
representation first introduced in \cite{D26}. The definition is motivated by by the self-adjoint operators $U_z^\nu$ given by $U_z^\nu f(w)=\kernu{z}(w)f(\tau_z(w)),w\in \B^n, f\in \Berg$, that were heavily used in the literature on operator theory.
In an upcoming article \cite{DDO25}, more details will be provided on the matter, including a comparison of these two different approaches, which makes no difference for the QHA framework. Thus, we provide proofs only when necessary and point to the results on the unweighted case in \cite{DDMO25} for details.

Define for $x\in \SUn$ an operator $\pinu(x)$ on $\Berg$ by
\begin{equation}\label{eq:Projective}
   \pinu(x)f(z):= \kernu{x\cdot 0}(z)f(x^{-1}\cdot z), \quad z\in \nball, f\in \Berg . 
\end{equation} 
Then $\pinu:G \rightarrow \U(\Berg)$ is a square-integrable irreducible projective unitary representation of $G$ acting on the Bergman space $\Berg$ and satisfies the following:
\begin{align}\label{eq:representation}
    \pinu(x)\pinu(y)=m_\nu(x,y)\pinu(xy),
\end{align}
where  $m_\nu(x,y)=\frac{|h(x\cdot 0,y\cdot 0)|^\nu}{h(x\cdot 0,y\cdot 0)^\nu}\in \T$. Note that the power is well defined because the right
hand side is a function on the simply connected domain $\nball \times \nball$.

\subsection{Translations and convolutions of operators}\label{sec:repr}
 The reader can find an excellent discussion about QHA in general in  \cite{Halvdansson2023} and we will use this article as a
general reference.   For an overview of QHA on unimodular groups, we point to \cite[Appendix A]{DDMO25}. Some additional interesting 
properties arise in the special case of $G=\mathrm{SU}(n,1)$ acting on the Bergman space $\Berg$, as discussed in \cite{DDMO25} for the unweighted case.
We do not present proofs of what follows and point to \cite{Halvdansson2023} and \cite{DDMO25} for details.
 
For $A\in \bdd$, we define the translation of $A$ by $x\in G$, denoted $\actopL{x}{A}$, by
$$\actopL{x}{A}=\pinu(x)A\pinu(x)^*,\quad x\in G.$$

Following the general discussion in the beginning of Section \ref{sec:L1K},
we define the convolution $f\ast A$ of a fucntion $f:G\to \C$ and $A\in \bdd$, formally by the weak integral
$$f\ast A= \int_G f(x) \actopL{x}{A} \haar{x}.$$
When the operator $f\ast A$ is well-defined, we have
$$\ip{(f\ast A)g_1}{g_2}=\int_G \psi(x) \ip{\actopL{x}{A}g_1}{g_2} \haar{x}, \quad g_1,g_2\in \Berg.$$ 

Let $1\leq p\leq \infty$. Denote the space of Schatten-$p$ class operators on $\Berg$ by $\schatten$,  and its norm by $\|\cdot\|_p$. 
When $p=\infty$, we make the identification $\cT^\infty(\Berg)=\bdd$.
We then have the following QHA Young's inequalities:
\begin{enumerate}
    \item If $f\in \LoneG$ and $A\in \schatten$, then $f\ast A\in \schatten$ and $\|f\ast A\|_{\cT^p}\leq \frac{1}{c_\nu}\|f\|_1\|A\|_{\cT^p}$.
    \item If $f\in \LpG$ and $A\in \traceop $, then $f\ast A\in \schatten$ and $\|f\ast A\|_{\cT^p}\leq \frac{1}{c_\nu}\|f\|_1\|A\|_{\cT^p}$.
\end{enumerate}
When the convolutions are well defined by the QHA Young's inequalities, we have the following associativity for convolutions of functions $f,g:G\to \C$, and $A\in \bdd$:
$$f\ast (g\ast A)=(f\ast g)\ast A.$$

For $A,B\in \traceop$, we define the convolution $A\ast B:G\to \C$ by
$$(A\ast B)(x)=\tr(A\actopL{x}{B}),\quad x\in G.$$
The QHA Young's inequalities for the above convolution are given by:
\begin{equation}
    \|A\ast B\|_{\cT^p}\leq \|A\|_{\cT^1}\|B\|_{\cT^p}, \quad A\in \traceop, S\in \schatten.
\end{equation}

In the following proposition, we note that Toeplitz operators and Berezin transforms are QHA convolutions. We present it without proof, for a similar proof we
 point to \cite[Lemma 3.2]{D26} and \cite[Lemma 3.9]{DDMO25}. Recall the function $\hnu(z)=(1-|z|^2)^\nu$ \eqref{eq:hnu} and the rank one  projection $P_0=1\otimes 1$.

\begin{lemma}\label{lem:ToepBerezin convol}
    Let $1\leq p\leq \infty$. Let $f\in \Lpinv$ and let $S\in \schatten$. Then 
    \begin{enumerate}
        \item $\Toep_f=c_\nu(f\ast P_0)$
        \item $B^{(\nu)}(S)=S\ast P_0$
        \item $\hnu= \Berz(P_0)= (P_0\ast P_0)$
    \end{enumerate}
    Moreover, $\Toep_a\in \schatten$ and $B^{(\nu)}(S)\in \Lpinv$.
\end{lemma}

The following lemma describes adjoints of QHA convolutions. Recall that the usual involution of a function $f:G\to \C$ is given by $f^*(x)=\overline{f(x^{-1})}$. Note that for radial functions on the ball $\B^n$, this is the same as the complex conjugate by Lemma \ref{lem:radial_functions}.
\begin{lemma}\label{lem:adjoints}
    Let $f:G\to \C$, and let $A, B\in \bdd$ s.t. the convolutions given below are well-defined. Then
    \begin{enumerate}
        \item $(f\ast A)^*=f^*\ast A^*$
        \item $(A\ast B)^*=B^*\ast A^*$.
    \end{enumerate}
\end{lemma}
\begin{proof}
    Note that for $g_1,g_2\in \Berg$, we have
    \begin{align*}
        \ipnu{(f\ast A)^*g_1}{g_2}&= \ipnu{g_1}{(f\ast A)g_2}\\
        &=\overline{\ipnu{(f\ast A)g_2}{g_1}}\\
        &= \int_G \overline{f(x)}\overline{\ipnu{\actopL{x}{A}g_2}{g_1}} \haar{x}\\
        &= \int_G \overline{f(x)} \ipnu{\actopL{x^{-1}}{A^*}g_1}{g_2} \haar{x}\\
        &= \int_G \overline{f(x^{-1})} \ipnu{\actopL{x}{A^*}g_1}{g_2} \haar{x}\quad \text{(as $G$ is unimodular)}\\
        &= \ipnu{(f^*\ast A^*)g_1}{g_2},
    \end{align*}
    proving the first identity. Also, note that
    \begin{align*}
        (A\ast B)^*(x)&=\overline{(A\ast B)(x^{-1})}\\
        &= \overline{\tr(A\actopL{x^{-1}}{B})}\\
        &= \tr((A\actopL{x^{-1}}{B})^*)\\
        &=\tr(B^*\actopL{x}{A^*})\\
        &=(B^*\ast A^*)(x),
    \end{align*}
    which shows the second identity.
\end{proof}

We also note that convolutions preserve nonnegativity.
\begin{lemma}
    Let $A \geq 0$ be trace class and $B \geq 0$ be bounded. Then, $A \ast B \geq 0$. 
\end{lemma}
\begin{proof}
As every trace class operator can be written as $A = \sum_{j=0}^\infty c_j u_j\otimes \bar u_j$
where $(u_j)_j$ is a orthonormal  basis and $\tr A = \sum_{j=0}^\infty$.    Without loss of generality, it therefore
suffices to prove this for $A$ being of rank one, say $A = f \otimes f$. Then,
    \begin{align*}
        A \ast B(x) = \tr(A \pi_\nu (x) B \pi_\nu (x) ^\ast) = \langle B\pinu(x)^\ast f, \pinu(x)^\ast f\rangle \geq 0
    \end{align*}
    as required.
\end{proof}

\begin{example} Some well-known objects from harmonic analysis on Lie groups fall into this line of work. First, we note that
if $A\in\cT^p$ and we take $B=\id$, then $A\ast B$ is constant and equals $\tr A$. In particular, if $f\in C_c^\infty (G)$ then by
well known results by Harish-Chandra 
$\pi_\nu (f)$ is a trace class operator and $\Theta_\nu (f) = \tr ~\pi_\nu (f)$ defines a distribution on $G$. In our language, we now have 
\[ \Theta_\nu (f) =\id \ast \pi_\nu (f) .\]
\end{example}
 
\subsection{Radial operators and their convolutions}  The group $\SUn$ is not commutative, and hence the convolution is not commutative. But, as mentioned in Section \ref{sec:L1K}, the convolution algebra $\LoneinvK$ is commutative. We now note that an analogous commutativity property holds for convolutions of radial operators. This observation will be used in Section \ref{sec:convolalg} to prove that the algebra
$\fA_\nu=\LoneinvK\oplus \traceopK $ is commutative, see definition \eqref{eq:Anu}. 
Recall the maximal compact subgroup  $K\simeq \rU(n)$ from \eqref{K}. For $k\in K$, the operators $\pinu(k)$ are simple given by
$$\pinu(k)f=\actfL{k}{f},\quad f\in\Berg.$$
We say $A\in \bdd$ is radial if 
$$\actopL{k}{A}=A, \quad \forall k\in K$$
 which is equivalent to $A$ being a $K$-intertwining operator, $\pi_\nu (k) A = A\pi_\nu (k)$.
For $A\in \bdd$ 
the so-called radialization of $A\in \bdd$ is given by
$$A^K= \radop{A}=\int_K \actopL{k}{A} \haarK{k}.$$

\begin{lemma} If $f$ is right $K$-invariant, i.e., a function on the ball, then $f\ast A = f\ast A^K$. If $A$ is radial, then
$f\ast A =f_K \ast A$ where $f_K(x) = \int_K f(xk)d\mu_K(k)$ is the projection of $f$ onto a function on $\nball$.
\end{lemma}
\begin{proof} We have for $k\in K$ using that the multiplier is trivial on $K$:
\begin{align*}
f\ast A &= \int f(x) \pi_\nu (x) A \pi_\nu (x)^* d\mu_G(x)\\
& = \int  f(xk) \pi_\nu (x) A \pi_\nu (x)^* d\mu_G(x)\\
&= \int f(x) \pi_\nu (xk) A \pi_\nu (xk)^* d\mu_G(x)
\\
&= \int f(x) \pi_\nu (x) \pi_\nu (k) A \pi_\nu (k)^* \pi_\nu (x) d\mu_G(x)
\end{align*}
and now integrate over $K$. The other statement follows in the same way by repeating this argument backwards.
\end{proof}

The following is well-known:

\begin{lemma}
    Let $A\in \bdd$. Then the following are equivalent:
    \begin{enumerate}
        \item $A$ is radial;
        \item The Berezin transform $\Berz(A)$ is a radial function;
        \item $\radop{A}=A$.
    \end{enumerate}
\end{lemma}

The representation $\pinu$ restricted to $K$ is no longer irreducible, and admits a multiplicity-free decomposition into irreducible subrepresentations $\pinu^m|_K$, $m\in\N_0$, 
where $\pinu^m$ is the restriction of $\pinu|_K$ to  the finite-dimensional space
$\rP^m(\C)$ of homogeneous polynomial of degree $m$. Moreover,   as mentioned in Theorem \ref{lem:ProAa}, we have the orthogonal direct sum
$$\Berg=\bigoplus_{m=0}^\infty \rP^m(\C ),$$
 The representation of $K$ on $\rP^m(\C)$ is irreducible. Hence, 
as $A$ is radial if and only if $A$ is a $K$-intertwining operator,  it follows from Schur's Lemma that:

\begin{lemma} An operator $A\in \bdd$ is radial if only if $A= \sum_{m=0}^\infty c(A,m)P_m$ where $P_m=\sum_{|\vm|=m} f_\vm \otimes f_\vm$ is the orthogonal projection
onto $\rP^m(\C)$ and $c(A,m)\in \C$. Moreover, the following holds, where $d_m = \dim \rP^m(\C)$:  
 Then
\begin{enumerate}
\item $\| A\| =\sup_{m\in\N_0} |c (A,m)|$.
\item $A$ is trace class if and only if $\sum_{m=0}^\infty d_m|c(A,m)| < \infty$ and in that case
$\Tr A = \sum_{m=0}^\infty d_mc(A,m) $.
\item $A\in\schatten $, $p\not= \infty$ if and only if $\sum_{m=0}^\infty d_m |c(A,m)|^p<\infty$.
\end{enumerate}
\end{lemma}
\begin{corollary} If $p\not= \infty$ then $\schatten$ is isomorphic to $\ell^p$ where the
isomorphism is given by $A\mapsto (d_m^{1/p} c(A,m))_{m=0}^\infty$.
\end{corollary}

The following lemma is Corollary 3.3 and Lemma 3.11 in \cite{DDMO25}, generalized to weighted Bergman spaces. 
\begin{lemma}\label{lem:radial_convol}
    Let $A\in \bdd$ and $B,C\in \traceop$. Then:
     \begin{enumerate}
        \item If $f\in\LoneinvK$ and $A$ is radial, then $f\ast A$ is radial.
        \item If  $A$ and $B$ are  radial, then $A\ast B$ is a radial function on $\B^n$.     
        \item If $A$ and $B$ are radial, then $A\ast B = B\ast A$.
        \item If $B$ and $C$ are radial, then $(A\ast B)\ast C= (A\ast C)\ast B$.   
        \item If $A$, $B$, and $C$ are radial, then $(A\ast B)\ast C= (B\ast C)\ast A$.   
    \end{enumerate}
\end{lemma}
\begin{remark}
    In \cite{DDMO25}, the above result was established only for the unweighted Bergman space. However, since the proof carries over verbatim to the weighted Bergman spaces, we omit the details.
\end{remark}

It is well-known that the adjoint of the Toeplitz operator $\Toep_f$ with symbol $f$ is the Toeplitz operator $\Toep_{\Bar{f}}$. This is a special case of the following lemma.

\begin{lemma}
    Let $f:\B^n\to \C$ and $A\in \bdd$ be a radial operator be such that the convolution $f\ast A$ is well defined. Then $(f\ast A)^*=\Bar{f}\ast A^*$.
\end{lemma}
\begin{proof}
    Let $x\in G$. Then by Lemma \ref{lem:Aut}, there is $k_x\in K$ s.t. $x=t_{x\cdot0}k_x$. 
    Then
    $$\actopL{x}{A}=\actopL{t_{x\cdot 0}}{\actopL{k_x}{A}}=\actopL{t_z}{A}=\pinu(t_z) A \pinu(t_z) $$
    where $\pi_\nu(t_z)=U_z^\nu$ are the self-adjoint operators given by
    $$\pinu(t_z)f (w)=\kernu{z}(w)f(\tau_z(w)),\quad w\in \B^n, f\in \Berg,$$
    and $\tau_z\in \Aut(\B^n)$ is the involution that interchanges $z$ and $0$.
    Then for $g_1,g_2\in \Berg$,
    \begin{align*}
        \ipnu{(f\ast A)^*g_1}{g_2} &= \int_G \overline{f(x)} \overline{\ipnu{g_1}{\actopL{x}{A}g_2}} \haar{x}\\
        &=\int_{\B^n} \overline{f(z)} \overline{\ipnu{g_1}{\actopL{\tau _z}{A}g_2}} \inv{z}\\
        &=\int_{\B^n} \overline{f(z)} \ipnu{\actopL{\tau_z}{A^*}g_1}{g_2} \inv{z}\\
         &=\int_G \overline{f(x)} \ipnu{\actopL{x}{A^*}g_1}{g_2} \haar{x}\quad \text{(as $A^*$ is also radial)}\\
        &=\ipnu{(\Bar{f}\ast A^*)g_1}{g_2}
    \end{align*}
    as required.
\end{proof}

\subsection{Duality} 
Recall that the Banach space dual of $L^p(\SUn)$, $p\not= \infty$ is $L^{p^\prime}(\SUn)$, $\frac{1}{p}+ \frac{1}{p^\prime}$,
where the duality is given by 
$$\ip{f}{g}_{L^1}:=\int_{G}f(x)g(x)\haar{x} \qquad f \in L^p(G), g \in L^{p'}(G).$$

Meanwhile, the dual of $\cT^p(\Berg)$ is $\cT^{p^\prime}(\Berg)$ via the trace pairing: 
$$\ip{A}{B}_\tr:=\tr(AB), \quad A\in \cT^p(\Berg) ,  B\in \cT^{p^\prime}(\Berg).$$
In particular, the dual of $\cT^1(\Berg)$ is $\bdd$.
We have the following duality relation between radial convolutions that plays a role in Section \ref{sec:convolalg}. For a proof, 
we point to \cite[Lemma A.17]{DDMO25}
\begin{lemma}\label{lem:duality}
    Let $f\in \LoneG$, $A\in \traceop$ and $B\in \bdd$. Then $f\ast A\in \traceop$, $A\ast B\in \bddfG$, and the following is satisfied: 
    $$\ip{f\ast A}{B}_\tr=\ip{f}{A\ast B}_{L^1}.$$
\end{lemma}

Now we describe the dual of radial trace class operators $\traceopK$  which fits well with the identification of $\cT^p(\Berg)$ with $\ell^p$. The proof of the statement is relatively straightforward, but we include it for completeness: 

\begin{proposition}\label{prop:dual_of_radial_trace_class}
    By trace duality, we can identify $(\traceopK)' \cong \bddK$. Further, for $A \in \bddK$, we have
    \begin{align*}
        \| A\| = \sup_{B \in \traceopK, ~\| B\| \leq 1} |\tr(AB)|.
    \end{align*}
\end{proposition}
\begin{proof}  
    Clearly, every $A \in \bddK$ induces a bounded linear functional on $\traceopK$ through the trace. Let $\Phi \in (\traceopK)'$. Further, denote by $\rad: \bdd \to \bddK$ the radialization operator. Then, 
    \begin{align*}
        \Phi_1(B) := \Phi(\radop{B})
    \end{align*}
    is a bounded extension of $\Phi$ to all of $\traceop$. Since $(\traceop)' \cong \bdd$ by trace duality, there exists $A \in \bdd$ such that 
    \begin{align*}
        \Phi_1(B) = \tr(AB).
    \end{align*}
    But $\Phi_1(B) = \Phi_1(\radop{B})$ and hence (using that $K$ is unimodular):
    \begin{align*}
        \Phi_1(B) &= \tr(A~\radop{B}) = \int_K \tr(A\pi(k) B \pi(k^{-1}) \haarK{k} = \int_K \tr(\pi(k) A \pi(k^{-1}) B) \haarK{k}\\
        &= \tr(\radop{A} B).
    \end{align*}
    Therefore, for radial $B$, $\Phi(B) = \tr(AB)$ for some $A \in \bddK$. Finally, regarding the norm, we see that for $A \in \bddK$ (as the standard trace duality is also inducing the operator norm on $\bdd$):
    \begin{align*}
        \| A\| &= \sup_{B \in \traceop, ~\| B \| \leq 1} |\tr(AB)|\\
        &= \sup_{B \in \traceop, ~\| B\| \leq 1} |\tr(\radop{A}B)|\\
        &= \sup_{B \in \traceop, ~\| B\| \leq 1} |\tr(A~ \radop{B})|\\
        &= \sup_{B \in \traceopK, ~\| B\| \leq 1} |\tr(AB)|.
    \end{align*}
    Here, we used in the last step that $\| \radop{B}\|_{\traceop}\leq \| B\|_{\traceop}$ for each trace class $B$ and that $\radop{B} = B$ for each $B\in \traceopK$.
\end{proof}


\section{The convolution algebra and its spectrum}\label{sec:convolalg}

The main objective of this section is to understand the structure of the commutative convolution algebra $\convolalg=\convolalgoplus$. This algebra, formed by combining radial integrable functions with radial trace-class operators, provides a natural extension of the classical convolution algebra $\LoneinvK$. We describe the Gelfand Spectrum of $\convolalg$, and clarify how multiplicative functionals act within this setting. 

 \subsection{The algebraic structure of $\convolalg$}
The convolution algebra $\convolalg$ is defined to be
\begin{equation}\label{eq:Anu}
\convolalg:=\convolalgoplus
\end{equation}
with the norm $\|(f,A)\| = \|f\|_1 + \|A\|_1$ and multiplication and $*$-involution are given by
\begin{equation}\label{eq:Prod}
  (f, A) \ast (g, B) := (f \ast g + A \ast B, f \ast B + g \ast A ),\quad\text{and}\quad (f,A)^* = (f^*,A^*).
\end{equation}
Our first goal is to determine the Gelfand spectrum $\maxid$ of $\convolalg$, i.e., the space of (non-zero) multiplicative 
linear functionals on $\convolalg$. Later on, we will concentrate on the $*$-homomorphisms.

\begin{proposition}
    For $f, g \in \LoneinvK$ and $A, B \in \traceopK$ the following holds true:
    \begin{enumerate}
        \item $f \ast g \in \LoneinvK$;
        \item $f \ast A \in \traceopK$;
        \item $A \ast B \in \LoneinvK$;
        \item $\fA_\nu$ is a commutative Banach $*$-algebra.
    \end{enumerate}
\end{proposition}

\begin{proof}
    First statement is well-known and (2) and (3) follow from Lemma \ref{lem:radial_convol}. Commutativity of $\convolalg$ follows from the fact that $f\ast g$ and $A\ast B$ are commutative (Lemma \ref{lem:radial_convol}). Also, with the given involution $\convolalg$ is a $*$-algebra by Lemma \ref{lem:adjoints}. Finally, we have
\begin{align*} \|(f,A)\ast (g,B)\| & \le \| f\ast g+A\ast B\|_1 + \|f\ast B + g\ast A\|_1\\
& \le \|f\|_1 \|g\|_1 + \|A\|_1 \|B\|_1 + \|f\|_1\|B\|_1+\|g\|_1\|A\|_1\\
&= (\|f\|_1+\|A\|_1)(\|g\|_1+\|B\|_1).\qedhere
\end{align*}
\end{proof}

From now on we will identify $\LoneinvK$ with the subalgebra $\{(f,0)\mid f\in \LoneinvK\}$ using the the map $f\mapsto (f,0)$. Similarly
we identify $\traceopK$ with the subspace $\{(0,A )\mid A\in\traceopK)$ via $A\mapsto (0,A)$. We note that
\[(f,0)*(g,0) = (f*g,0)\quad \text{and}\quad (f,0)^* = (f^*,0)\]
so the above identification is a $*$-isomorphism onto a $*$-subalgebra of $\convolalg$. Furthermore $(0,S)*(0,T) = (S*T,0)$ or
$\traceopK*\traceopK\subset \LoneinvK$.

\subsection{The space of multiplicative linear functionals on $\mathfrak A_\nu$}
We recall that $\LoneinvK$ is itself already a commutative Banach $*$-algebra. It is well-known that its Gelfand theory can be described in terms of the \emph{bounded spherical functions}, cf.~Appendix \ref{app:Gelfandpairs} for a discussion on general Gelfand pairs. For the concrete setting we consider, these functions can be explicitly computed (cf.~Section \ref{sec:spherical_functions}). As an outcome, it can be shown that, as topological spaces, the Gelfand space $\mathcal M(\LoneinvK)$ can be very explicitly computed (cf.~Theorem \ref{thm:gelfand_space_function_algebra}): $\mathcal M(\LoneinvK) \cong \spheri \cong (\mathbb R \times i[-n, n]) / \sim$, where $\sim$ is the equivalence relation of antipodal identification.

The Gelfand transform of $\LoneinvK$ can very conveniently be written in terms of the spherical Fourier transform: Identifying bounded spherical functions $\phi = \phi_\lambda$ with points in $(\mathbb R \times i[-n, n]) / \sim$ (cf.~Theorem \ref{thm:spherical_functions}), the spherical Fourier transform of a function $f \in \LoneinvK$ is given by
\begin{align*}
    \widehat{f}(\lambda) := \widehat{f}(\phi_\lambda) := \int_G f(x)\phi_\lambda (x^{-1}) \haar{x}, \quad \lambda \in (\mathbb R \times i[-n, n]) / \sim.
\end{align*}

We want to turn to the description of the Gelfand spectrum $\maxid$ of $\convolalg$. 
We will typically denote elements of $\maxid$ by uppercase Greek letters such as $\Phi$. We adopt the notations $\Phi(f):=\Phi(f,0)$ and $\Phi(A):=\Phi(0,A)$, for $f\in \LoneinvK$ and $A\in \traceopK$.
We note that $\Phi\in \maxid$ induces a multiplicative functional on $\LoneinvK$. The next lemma is a consequence of this. 
 
\begin{lemma}\label{lem:functional_on_functions} Let $\Phi\in\maxid$. Then there exists a unique $\phi_{\lambda ( \Phi)}
\in \spheri$ such that
\[\Phi (f,0) = \wf (\lambda) = \int_G f(x)\phi_{\lambda (\Phi )} (x)dx, \quad  f \in \LoneinvK. \]
\end{lemma}
  
We continue with the following lemma:
\begin{theorem}\label{lem:regularity} For $\Phi \in \maxid$ let $\phi_{\lambda }$ denote the corresponding
spherical function. Then 
\[\Phi (P_0)^2 = \widehat{h_\nu}(\lambda) = B_\nu (\phi_\lambda ) (0) \not= 0 .\]
In particular, there exists a square root of $\widehat{h_\nu}(\lambda)$ on $\spheri$ such that 
\begin{enumerate}
\item $\lambda \mapsto \sqrt {\widehat{h_\nu}(\lambda)}$ is continuous on $\spheri \cong (\R \times i[-n,n])/ \sim$,
\item $\Phi (P_0 ) = \pm \sqrt {\widehat{h_\nu}(\lambda)}$.
\end{enumerate} 
\end{theorem}
\begin{proof} Recall that $h_\nu =P_0\ast P_0$.  Using Lemma \ref{lem:Bnvarphi},  we get
\[\Phi (P_0)^2 = \widehat{h_\nu}(\lambda) =B_\nu (\phi_\lambda)(0) \not= 0 .\]
As $\lambda \mapsto  \widehat{h_\nu}(\lambda)$ is continuous on $\spheri$ and $\spheri$ is simply connected by Corollary \ref{cor:simply_connected} and $\hat h_\nu (\lambda )\not=
0$, it follows that there exists a continuous square root of $\hat h_\nu (\lambda)$ on  $\spheripd$. 
\end{proof}
\begin{remark}
    If one prefers, one could instead work with even functions on $\mathbb R \times i[-n, n]$ instead of functions on $\spheri \cong (\mathbb R \times i[-n, n])/\sim$. In this case, the resulting square root would turn out to be an even function holomorphic on $\mathbb R \times i(-n, n)$ and continuous up to the boundary.
\end{remark}
Throughout the following, we will always fix a continuous square root $\sqrt{\widehat{\hnu}}$ of $\widehat{\hnu}$. The following result is now tautological:
\begin{lemma}\label{lem:functionalPzero}
    For  $\Phi \in \maxid$ and $\phi_\lambda $ as in Lemma \ref{lem:functional_on_functions} we have
    \begin{align*}
        \Phi(P_0) = \pm \sqrt{\widehat{\hnu}(\phi_\lambda)}.
    \end{align*}
\end{lemma}

We now describe how the multiplicative functionals act on operators in terms of the spherical Fourier transform of functions.

\begin{proposition}\label{prop:functional_characterize}
  Fix a continuous square root of $\widehat{h}_\nu$.  For $\phi_\lambda \in \spheri$ and $j \in \{ 0, 1\}$, define $\Phi_{\lambda, j}$ by
    $$\Phi_{\lambda, j}(f,A):=\widehat{f}(\lambda ) +(-1)^j \frac{\widehat{(A\ast P_0)}(\lambda)}{\sqrt{\widehat{\hnu}(\lambda)}},
    \quad (f,A)\in \convolalg.$$
    Then
    \begin{enumerate}
        \item  $\Phi_{\lambda, j}\in \maxid$ for all $\lambda \in \spheri$ and $j = 0, 1$.
        \item If $\Phi\in \maxid$ then $\Phi=\Phi_{\lambda, j}$ for some $\lambda\in (\R \times i[-n,n])/\sim$ and $j = 0, 1$.
        \item For $\lambda \in (\R \times i[-n,n])/\sim$, define the operator $T^{\lambda,\nu}$,  to be 
        \[T^{\lambda,\nu} =\widehat{\hnu}(\lambda)^{-1/2}\, \Toep_{\phi_\lambda} ,\]
where $\Toep_{\phi_\lambda}$ is the Toeplitz operator with symbol $\phi_\lambda$. Then 
        $$\Phi_{\lambda, j}(f,A)=\widehat{f}(\lambda)+(-1)^j\tr(AT^{\lambda,\nu}), \quad (f,A)\in \convolalg.$$
        
 \item  $T^{\lambda,\nu}$ is a radial operator whose $c_m(\lambda,\nu)$  eigenvalue on $\rP^m(\C)$  is given by
\[ c_m (\lambda ,\nu )  = \frac{\Gamma(m +\nu)}{\widehat{h_\nu}(\lambda )^{1/2} \Gamma (\nu - n)m! (n-1)!} 
\int_0^1 \phi_\lambda (\sqrt{r}) r^{n+m -1} (1-r)^{\nu - n -1} dr.\]
    \end{enumerate}
\end{proposition}
\begin{proof}
    We first verify (1): It is clear that $\Phi_{\lambda, j}$ is a linear functional on $\convolalg$ whose restriction to $\LoneinvK$ is multiplicative. We have to verify that it is multiplicative with respect to the convolutions of functions and operators.
    Let $A,B\in \traceop^K$. 
    \begin{align*}
        \Phi_{\lambda, j}((A\ast B)\ast \hnu)&=  \Phi_{\lambda, j}((A\ast B)\ast (P_0\ast P_0))\\
        &=  \Phi_{\lambda, j}((A\ast P_0)\ast (B\ast P_0))\\
        &= \widehat{(A\ast P_0)}(\lambda)\widehat{(B\ast P_0)}(\lambda).
    \end{align*}
    But also,
    \begin{align*}
        \Phi_{\lambda, j}((A\ast B)\ast \hnu)&=\widehat{(A\ast B)}(\lambda)\widehat{\hnu}(\lambda).
    \end{align*}
    Hence, we have
    $$\Phi_{\lambda, j}(A\ast B)=\Phi_{\lambda, j}(A)\Phi_{\lambda, j}(B).$$
    Also, for $f\in \LoneinvK$, we have
    \begin{align*}
        \Phi_{\lambda, j}(f\ast A)&=(-1)^j\frac{(\widehat{f\ast A\ast P_0})(\lambda)}{\sqrt{\widehat{\hnu}(\lambda)}}\\
        &=(-1)^j\frac{\widehat{f}(\lambda)\widehat{(A\ast P_0)}(\lambda)}{\sqrt{\widehat{\hnu}(\lambda)}}\\
        &=\Phi_{\lambda, j}(f)\Phi_{\lambda, j}(A).
    \end{align*}
    This concludes proving that $\Phi_{\lambda, j}$ is a multiplicative functional.
    
    To prove (2), let $\phi_\lambda$ be as in Lemma \ref{lem:functional_on_functions}. 
    Then for $f\in \LoneinvK$, $\Phi(f)=\widehat{f}(\lambda )$ and for $A\in\traceopK$, using that $\Phi (P_0)\not= 0$:
    $$\Phi(A)=\frac{\Phi(A)\Phi(P_0)}{\Phi(P_0)}=\frac{\Phi(A\ast P_0)}{\Phi(P_0)}=
    \pm\frac{\widehat{A\ast P_0}(\lambda )}{\sqrt{\widehat{\hnu}(\lambda)}}$$
    and the result follows from Lemma \ref{lem:functionalPzero}.
    
    For verifying (3), note that by the duality identity in Lemma \ref{lem:duality},
    $$\widehat{A\ast P_0}(\lambda )=\ip{A\ast P_0}{\phi_\lambda }_{L^1}=\ip{A}{\phi_\lambda \ast P_0}_\tr=\tr(A\Toep_{\phi_\lambda}).$$

 The formula for the eigenvalue is know, see \cite[Section 5]{Grudsky_etal2003}, but we include it here for completeness. 
 We recall from \eqref{defMea} that the measure $\mu_\nu$ is given by 
 \[ d\mu_\mu (z) =c_\nu (1-|z|^2)^{\nu - n -1} dz = c_\nu 2n (1-|z|^2)^{\nu -n -1}r drd\sigma (w)\]
 where $\sigma $ is the $K$-invariant probability measure on $\rS^{2n-1}$. We therefore get, using that $T^{\lambda,\nu} |_{\rP^m(\C)}$ is
 a constant:
 \begin{align*} \widehat{h_\nu}(\lambda )^{1/2} \|z_1^m\|_\nu^2 c_m (\lambda ,\nu) &= \ip{\Toep_{\phi_\lambda}z_1^m}{z_1^m}\\
 &= \int_{\nball}\int_{\nball} \phi_\lambda (w)w_1^m K(z,w) \bar{z_1^m}d\mu_\nu(w)d\mu_\nu (z)\\
 &=  \int_{\nball} \phi_\lambda (w)w_1^m \overline{\int_{\nball} K(w,z) z_1^m d\mu_\nu(z) }d\mu_\nu (w)\\
 &= \int_{\nball } \phi_\lambda (w)|w_1|^{2m}d\mu_\nu (w)\\
 & = c_{\nu}  2n\int_0^1 \phi_\lambda (r) r^{2m + 2n -2} (1-r^2)^{\nu - n -1} rdr\\
 &= c_{\nu} n \int_0^1 \phi_\lambda (\sqrt{s}) s^{n+m -1} (1-s)^{\nu - n -1} ds\qquad s= r^2 .
  \end{align*}
The claim now follows as $c_\nu = \Gamma (\nu)/(n! \Gamma (\nu -n)) $ and $\|z_1^m\|_\nu^2 = m! \Gamma (\nu)/\Gamma (\nu +m)$.
\end{proof}

\begin{remark}
\begin{enumerate}
    \item The above results depends on the choice of the square root $\sqrt{\widehat{\hnu} (\lambda )}$. Nevertheless, passing from one possible choice of a square root the other goes along with replacing $j$ by $j+1$ (modulo 2) in all the above formulas, so the actual choice is not really important.
\item One can use part 2 of Theorem \ref{thm:spherical_functions} to evaluate the constant further, but as far as we can see the result does not clarify the
situation at all so we do not include that calculation here.\end{enumerate}
\end{remark}

\begin{corollary}
    We have, as sets, the following identification of the Gelfand space:
    $$\maxid = (\mathbb R \times i[-n, n])/\sim ~ = \spheri.$$
\end{corollary}

\begin{lemma}
    Let $\lambda \in (\mathbb R \times i[-n, n])/\sim$. Then $\|T^{\lambda,\nu}\|\leq 1$.
\end{lemma}
\begin{proof}
We write $T=   T^{\lambda,\nu}$.   Boundedness follows by properties of multiplicative functionals and Proposition \ref{prop:dual_of_radial_trace_class}.
Using that   $\|\Phi \|=1$ we get:
    \begin{align*}
        \| T \| = \sup_{A \in \traceopK, \| A\|_1 \leq 1} |\tr(AT) | = \sup_{A \in \traceopK, \| A\|_1 \leq 1} |\Phi_{\lambda ,j}
         (0,A)|\leq \sup_{A \in \traceopK, \| A\|_1 \leq 1} \|A\|=1,
    \end{align*}
    which concludes the proof.
\end{proof}
This lemma, in particular, implies the following bound on the norm of Toeplitz operators with bounded spherical functions as symbols: 
$$\|\Toep _{\phi_\lambda} \|\leq |\widehat{\hnu}(\lambda )|^{\frac{1}{2}}.$$
This estimate is better than the usual estimate of $\| \Toep_{\phi_\lambda} \| \leq \|\phi_\lambda   \|_\infty = 1$.

So far, we have described the maximal ideal space as a set, i.e., we have identified all multiplicative linear functionals. The maximal 
ideal space canonically has the structure of a topological space, by endowing it with the subspace topology of the weak$^\ast$ topology on the 
dual of $\convolalg$. 

We start with the following lemma. 
\begin{lemma}\label{lem:wstar-cont}
    The map $(\mathbb R \times i[-n, n])/ \sim \ni \lambda \mapsto T^{\lambda,\nu} \in \mathcal L(\Berg)^K$ is weak$^\ast$-continuous (with respect to the predual $\traceopK$).
\end{lemma}
\begin{proof}
We need to show that for every $A \in \traceopK$ the map
\begin{align*}
    (\mathbb R \times i[-n, n])/ \sim \ni \lambda \mapsto \tr(AT^{\lambda,\nu})
\end{align*}
is continuous. But this is an immediate consequence of the fact that
\begin{align*}
    \tr(AT^{\lambda,\nu}) = \frac{\widehat{(A \ast P_0)}(\lambda)}{\sqrt{\widehat{\hnu}(\lambda)}},
\end{align*}
where the right-hand side is a quotient of continuous functions.
\end{proof}

\begin{theorem}
    As topological spaces, we have
    \begin{align*}
        \mathcal M(\convolalg) \cong [(\mathbb R \times i[-n, n])/ \sim] \times \mathbb Z_2
    \end{align*}
    upon identifying the point $(\lambda, j) \in [(\mathbb R \times i[-n, n])/ \sim] \times \mathbb Z_2$ with the multiplicative linear functional
    \begin{align*}
        (f, A) \mapsto \Phi_{\lambda, j}(f, A) =  \widehat{f}(\lambda) + (-1)^j \tr(AT^{\lambda,\nu}).
    \end{align*}
\end{theorem}
\begin{proof}As usual, we will abbreviate $(\mathbb R \times i[-n, n])/ \sim] \cong \spheri$. 

    We already know that the identification in the theorem gives a 1:1 map between points in $\spheri \times \mathbb Z_2$ and multiplicative linear functionals on the algebra. To prove that this identification is also a topological isomorphism, we will show that a net $(\lambda_\gamma, j_\gamma)_{\gamma} \subset \spheri \times \mathbb Z_2$ converges in the Gelfand spectrum to $(\lambda, j) \in \spheri \times \mathbb Z_2$ if and only if it converges in the natural topology of $\spheri \times \mathbb Z_2$. Note that, by the definition of the weak$^\ast$ topology on the maximal ideal space, we have convergence $(\lambda_\gamma, j_\gamma) \to (\phi, j)$ if and only if
    \begin{align*}
        \Phi_{\lambda_\gamma, j_\gamma}(f, A) \to \Phi_{\lambda, j}(f, A)
    \end{align*}
    for every possible choice $(f, A) \in \convolalg$. Setting $A = 0$, we immediately see that we necessarily have
    \begin{align*}
        \widehat{f}(\lambda_\gamma) \to \widehat{f}(\lambda)
    \end{align*}
    for every $f \in \LoneinvK$. In particular, for $(\lambda_\gamma, j_\gamma) \to (\lambda, j)$ to hold true in the maximal ideal space, we necessarily have $\lambda_\gamma \to \lambda$ in $\spheri$.

    Let now $c > 0$ be small enough such that $\| c\hnu\|_{L^1} < 1$ (with respect to the invariant measure). As a side note, observe that for $n \geq 7$ we can simply let $c = 1$.
    By what we have seen before, this yields: $|\widehat{(c\hnu)}(\lambda)| < 1$. Therefore, we see that:
    \begin{align*}
        \Phi_{\lambda, j}(c\hnu, \sqrt{c}1\otimes 1) = \sqrt{\widehat{(c\hnu)}(\lambda)}(\sqrt{\widehat{(c\hnu)}(\lambda)} + (-1)^j)
    \end{align*}
    Since $|\sqrt{\widehat{ch_\nu}(\lambda)}| < 1$ for each $\lambda$, we see that for $(\lambda_\gamma, j_\gamma) \to (\lambda, j)$, it is necessary that eventually $j_\gamma = j$: When $j_\gamma = 0$, then the values of $\sqrt{\widehat{ch_\nu}(\lambda_\gamma)} + (-1)^{j_\gamma}$ are contained in a disk around $1$ with radius strictly smaller than one. If $j_\gamma = 1$, they are contained in a disk around $-1$ with radius strictly smaller than one.
    
    Hence summarizing, for convergence in the maximal ideal space it is necessary that $\lambda_\gamma \to \lambda$ in $\spheri$ and $j_\gamma \to j$ in the discrete topology. That this is also sufficient for convergence in the maximal ideal space follows from continuity of the Fourier transform of $L^1$ functions, together with the previous Lemma \ref{lem:wstar-cont}.
\end{proof}

\subsection{$^\ast$-homomorphisms of $\convolalg$}\label{sec:FTop}
The algebra $\fA_\nu$ comes also with a $*$-operator $(f,A) \mapsto (f,A)^* =(\bar f,A^*)$ where we have used that
$f^* = \bar f$. 
If $\Phi_{\lambda,j}$ is a $*$-homomorphism then $\Phi_{\lambda, j}|_{\LoneinvK}$ is a $*$-homomorphism and that happens if and only if 
$\phi_\lambda$ is positive definite, see \cite{vD09}, and by Theorem \ref{thm:spherical_functions}  that happens if and only if $\lambda \in (\R \cup i[-n,n])/\sim ~\cong \spheripd$.  

\begin{lemma}
    $\widehat{h_\nu}(\lambda) > 0$ for each $\lambda \in \spheripd$.
\end{lemma}
\begin{proof}
    By Theorem \ref{thm:spherical_functions}, each $\phi_\lambda$ with $\lambda \in (\R \cup i[-n,n])/\sim$ is real-valued. Since $h_\nu$ attains only positive values, we see that $\widehat{h_\nu}(\lambda)$ is real-valued for each $\lambda \in (\R \cup i[-n,n])/\sim$. Note that $\lambda = 0$ yields a positive definite spherical function, hence we can evaluate $\widehat{h_\nu}$ at $0$ to obtain $\widehat{h_\nu}(0) = \| h_\nu\|_{L^1} > 0$. Since $\spheripd$ is (path-)connected by Theorem \ref{thm:gelfand_space_function_algebra}, $\widehat{h_\nu}(\lambda)$ depends continuously on $\lambda$ and $\widehat{h_\nu}(\lambda) \neq 0$ for each $\lambda$ by Theorem \ref{lem:regularity}, we see that $\widehat{h_\nu}(\lambda) > 0$ for every $\lambda \in (\R \cup i[-n,n])/\sim$. 
\end{proof}
By the previous lemma, we can now consider by $\sqrt{\widehat{h_\nu}}$ always the choice of square root which takes only positive values when restricted to $\spheripd$. We assume this throughout the remainder of the paper.
\begin{lemma}
    Let $\Phi = \Phi_{\lambda, j} \in \maxid$. Then, $\Phi$ is a $^\ast$-homomorphism if and only if $\lambda \in (\R \cup i[-n,n])/\sim$.
\end{lemma}
\begin{proof}
    By restricting $\Phi$ to $\LoneinvK$, it is clear that $\lambda \in (\R \cup i[-n,n])/\sim$ is necessary. To show that this is also sufficient, we note that from Theorem \ref{thm:spherical_functions} we obtain that the positive-definite spherical functions are all real-valued. Therefore, we see that for $\lambda \in (\R \cup i[-n,n])/\sim$:
    \begin{align*}
        \Phi_{\lambda, j}(\overline{f}, A^\ast) ) = \widehat{\overline{f}}(\lambda) + (-1)^j \frac{1}{\sqrt{\widehat{h_\nu}(\lambda)}} \tr(A^\ast \Toep_{\phi_\lambda}).
    \end{align*}
    Clearly, $\widehat{\overline{f}}(\lambda) = \overline{\widehat{f}(\lambda)}$. Since $\phi_\lambda$ is real-valued, $\Toep_{\phi_\lambda}$ is self-adjoint such that
    \begin{align*}
        \tr(A^\ast \Toep_{\phi_\lambda}) = \tr((A \Toep_{\phi_\lambda})^\ast) = \overline{\tr(A\Toep_{\phi_\lambda})}
    \end{align*}
    and therefore
    \begin{align*}
        \Phi_{\lambda, j}(\overline{f}, A^\ast) = \overline{\Phi_{\lambda, j}(f, A)}.
    \end{align*}
    This finishes the proof.
\end{proof}
\begin{corollary}
    As a topological space, the space of $^\ast$-homomorphisms from $\convolalg$ to $\mathbb C$ agrees with $[(\R \cup i[-n,n])/\sim] \times \mathbb Z_2$.
\end{corollary}
    
\section{The Plancherel Formula}\label{sec:plancherel}
We recall from Theorem \ref{thm:spherical_functions} the measure $\mu_\fa$ on $[(\R \cup i[-n,n])/\sim]$ such that the spherical Fourier transform
extends to an unitary isomorphism $L^2(\nball )^K \to L^2([(\R \cup i[-n,n])/\sim], \mu_\fa)$. Since this measure is supported only on $[0, \infty) = \mathbb R^+ \cong (\mathbb R/\sim) \subset (\mathbb R \cup i[-n, n])/\sim$, we usually say that $\mu_\fa$ is a measure on $\mathbb R^+$ and formulate the Plancherel theorem for functions as an isomorphism $L^2(\nball)^K \cong L^2(\mathbb R^+, \mu_\fa)$. We now extend this result to $\fA_\nu$ so that the
Fourier-Gelfand transform becomes a unitary isomorphism. We only  have to deal with the operator part as the function part
is included in Theorem \ref{thm:spherical_functions}. We use here the topology coming from the Euclidean picture. We write
$\widehat A (\lambda,j)$ for $\Phi_{\lambda, j}(A)$. This leads us to the following definition:
\begin{definition}
    For $A \in \traceopK$ we define the \emph{operator Fourier transform} of $A$ as the function $\widehat{A}(\lambda) = \widehat A (\lambda,0) = \tr(AT^{\lambda,\nu})$.
\end{definition}
Note that $\widehat A (\lambda ,1) = - \widehat A (\lambda ,0)$ so the Fourier-Gelfand transform of operators is completely determined by 
$\lambda \mapsto \widehat A (\lambda)$.
 
\begin{lemma}[Riemann-Lebesgue lemma for operators] Let $A\in \traceopK$. Then $\widehat A \in C_0(\R)_{\text{even}} = C_0(\mathbb R^+)$ with $\| \widehat{A}\|_\infty \leq \| A\|_{tr}$. 
\end{lemma}
\begin{proof} We have $\widehat{ A^* \ast A}(\lambda ) = |\widehat{A}(\lambda )|^2$ and $A^*\ast A\in  L^1$.
It follows that $\widehat{ A^* \ast A}$ vanishes at infinity and hence $\widehat A$ does so too. The estimate for the norm is clear.
\end{proof}
\begin{remark}
    One even has $\widehat{A} \in C_0((\mathbb R \times i[-n, n])/\sim)$ by general Gelfand theory.
\end{remark}
\begin{lemma}
    Let $A \in\traceopK$. Then, $A^\ast \ast A$ is a continuous and positive definite function.
\end{lemma}
\begin{proof}
    Using the convolution formalism, the proof is relatively straightforward and similar to the proof of the fact that $f^\ast \ast f$ is a positive definite function. We provide it for completeness.

    Let $g_1, \dots, g_m \in G$ and $c_1, \dots, c_m \in \mathbb C$. Then:
    \begin{align*}
        \sum_{j, k = 1}^m A^\ast \ast A(g_j^{-1} g_k) \overline{c_j}c_k &= \sum_{j, k = 1}^m \tr(A^\ast \pi_\nu(g_j^{-1}g_k) A \pi_\nu(g_j^{-1}g_k)^\ast)\overline{c_j}c_k\\
        &= \sum_{j, k = 1}^m \tr(A^\ast\pi_\nu(g_j^{-1}) \pi_\nu(g_k) A \pi_\nu(g_k)^\ast \pi_\nu(g_j^{-1})^\ast)\overline{c_j}c_k\\
        &= \sum_{j, k = 1}^m \tr(\pi_\nu(g_j^{-1})^\ast A^\ast \pi_\nu(g_j^{-1}) \pi_\nu(g_k) A \pi_\nu(g_k)^\ast )\overline{c_j}c_k
    \end{align*}
    By \eqref{eq:representation} we have that $\pinu(g)^*=\frac{1}{m_\nu(g,g^{-1})}\pinu(g^{-1})=\frac{1}{m_\nu(g^{-1},g)}\pinu(g^{-1})$ for $g\in G$. Then it follows that
    $$\pi_\nu(g_j^{-1})^\ast A^\ast \pi_\nu(g_j^{-1})=\frac{1}{m_\nu(g_j^{-1},g_j)} \pinu(g_j)A^\ast \pinu(g_j^{-1})=\pinu(g_j)A^\ast \pinu(g_j)^*.$$
    Hence
        \begin{align*}
         \sum_{j, k = 1}^m A^\ast \ast A(g_j^{-1} g_k) \overline{c_j}c_k &= \sum_{j,k=1}^m \tr((\pi_\nu(g_j) A \pi_\nu(g_j)^*)^\ast (\pi_\nu(g_k)A\pi_\nu(g_k)^*))\overline{c_j}c_k\\
        &= \tr(|\sum_{j=1}^n c_j \pi_\nu(g_j) A \pi_\nu(g_j)^\ast|^2)\\
        &= \| \sum_{j=1}^n c_j \pi_\nu(g_j) A \pi_\nu(g_j)^\ast\|_{\mathcal{HS}}^2 \geq 0.
    \end{align*}
    The fact that $A^\ast \ast A$ is continuous is clear.
\end{proof}

\begin{lemma} Let $A\in \traceopK$. Then
\[\|A\|^2_2=\| \widehat A\|_2^2= \int_{\R^+} |\widehat{A}(\lambda)|^2 d\mu_\fa (\lambda ).\]
\end{lemma}
\begin{proof} The function $f = A^* \ast A$ is
positive definite and hence by the inverse spherical Transform (for example, by \cite[Theorem 9.4.1]{Wolf07}) and using that $\phi(e) = 1$ for every spherical function shows:
\[f(e)  = \int \widehat{f} (\lambda )d\mu_\fa (\lambda ) .\]
But we also have $\widehat f (\lambda )  = | \widehat A (\lambda )|^2$ so that:
\[f(e) = \tr |A|^2 = \|A\|^2_2 = \int_{\R^+} | \widehat A (\lambda )|^2 d\mu_\fa (\lambda ).\]
\end{proof}

We extend the measure $\mu_\fa$ to $\R^+ \times \Z_2$ by $\mu_\fa (\Omega \times  \{j\}) = \mu_\fa (\Omega  ) $, $\Omega\subset \R^+$ and $j = 0, 1$.
Then we can reformulate the above statement as (using the density of the trace class operators in the Hilbert-Schmidt operators). Note that $L^2(\mathbb R^+, \mu_\fa) \oplus L^2(\mathbb R^+, \mu_\fa) \cong L^2(\R^+ \times \mathbb Z_2, \mu_\fa)$ by letting 
\begin{align*}
    L^2(\mathbb R^+, \mu_\fa) \oplus L^2(\mathbb R^+, \mu_\fa) \ni (f, g) \mapsto \left [h(\lambda, j) = \begin{cases}
        f(\lambda), \quad j = 0\\
        g(\lambda), \quad j = 1
    \end{cases} \right] \in L^2(\R^+ \times \mathbb Z_2, \mu_\fa).
\end{align*}
Using this identification, we will use $L^2(\mathbb R^+, \mu_\fa) \oplus L^2(\mathbb R^+, \mu_\fa)$ and $L^2(\R^+ \times \mathbb Z_2, \mu_\fa)$ interchangeably in the following.

\begin{lemma}
The map $\cF : L^2(\nball)^K \oplus \mathcal T^2(\Berg)^K   \to L^2(\R^+\times \Z_2 ,\mu_\fa)$, $\cF(f, A) = (\widehat{f}, \widehat{A})$ is
an isometry.
\end{lemma}
One can show that the adjoint of the map has the following form:
\begin{lemma}
    The Hilbert space adjoint of $\cF$ is given by:
   \begin{align*}
        (\mathcal F)^\ast: &L^2(\R^+\times \Z_2 ,\mu_\fa) \to L^2(\nball)^K \oplus \mathcal T^2(\Berg)^K,\\
        (\mathcal F)^\ast(f_0, f_1) &= \left(\int_{\mathbb R^+} f_0(\lambda)\phi_\lambda~d\mu_\fa(\lambda), \int_{\mathbb R^+} f_1(\lambda)T^{\lambda,\nu} ~d\mu_\fa(\lambda)\right),
    \end{align*}
    whenever $(f_0, f_1) \in L^1(\R^+\times \Z_2 ,\mu_\fa)\cap L^2(\R^+\times \Z_2 ,\mu_\fa)$.
\end{lemma}
\begin{proof}
    For the first factor, this is just the classical statement for functions, so we focus on the part where operators occur. For $f = f_1 \in L^1(\R^+,\mu_\fa)\cap L^2(\R^+,\mu_\fa)$ and $A \in \mathcal T^1(\Berg)^K$ we note that:
    \begin{align*}
        \langle (\mathcal F)^\ast(f), A\rangle_{\mathcal T^2} &= \langle f, \mathcal F(A)\rangle_{L^2}\\
        &= \int_{\R^+} f(\lambda) \overline{\widehat{A}(\lambda)}~d\mu_\fa(\lambda)\\
        &= \int_{\R^+} f(\lambda) \tr(A^\ast (T^{\lambda,\nu})^\ast)~d\mu_\fa(\lambda).
    \end{align*}
    Now recalling that $\langle (\mathcal F)^\ast(f), A\rangle_{\mathcal T^2} = \tr((\mathcal F)^\ast(f) A^\ast)$, we see that (where we also use, as already observed earlier, that $T^{\lambda,\nu}$ is self-adjoint for $\lambda \in \R^+$):
    \begin{align*}
        \tr((\mathcal F)^\ast(f) A^\ast) &= \int_{\R^+}f(\lambda) \tr(A^\ast T^{\lambda,\nu})~d\mu_\fa(\lambda)\\
        &= \tr( A^\ast \int_{\R^+} f(\lambda) T^{\lambda,\nu}~d\mu_\fa(\lambda)).
    \end{align*}
    Since this holds true for all $A \in \mathcal T^1(\Berg)^K$, the claim follows.
\end{proof}
 
\begin{lemma}
    The range of $\cF|_{\mathcal T^2(\Berg)^K} : \mathcal T^2(\cA_\nu)^K \to L^2(\R^+,\mu_\fa)$ is dense in $L^2(\R^+,\mu_\fa)$.
\end{lemma}
\begin{proof}
    Again, it suffices to consider the range of the Fourier transform of operators, as it is well-known that the range of the spherical Fourier transform of functions is dense in $L^2(\R^+, \mu_\fa)$. Furthermore, it suffices to prove that every $g \in C_c( \R^+\times \Z_2)$ is contained in the closure of the range of the Fourier transform of operators. For this, fix such $g$ and let $f  = g/\widehat{P_0}$. Since $\widehat{P_0}$ is continuous and vanishes nowhere, we see that $f \in C_c(\R^+)$. Then, we can find a sequence $f_n \in L^2(\R^+)$
     such that $f_n \to f$ in $L^2$ norm and $f_n$ are in the range of the spherical Fourier transform of functions. Hence, there exist $f_n^\circ \in L^2(\nball)^K$ such that $\widehat{f_n^\circ} = f_n$. By the convolution formula, $\widehat{f_n^\circ \ast P_0} = f_n \cdot \widehat{P_0} \to g$ as $n \to \infty$ with convergence in $L^2(\R^+, \mu_\fa)$.
\end{proof}
Since the adjoint of a map with dense range has to be injective, we immediately obtain:
\begin{theorem}[Plancherel Theorem] The Fourier transform extends to an unitary isomorphism
$L^2(\nball)^K \oplus \mathcal T^2(\Berg)^K \simeq L^2(\R^+ \times \Z_2,\mu_\fa)$.
\end{theorem}
 
 \begin{remark} We note that we need both $j=0$ and $j=1$ because the image of $L^2(\nball )^K$ is already the full space $L^2(\R^+,\mu_\fa)$. 
 \end{remark}

\section{Further results on the Fourier transform of operators}\label{sec:fouriertrafo}
In this section, we discuss properties of the spherical Fourier transform of operators that go beyond Plancherel's theorem.
The first observation is that, just judging by the formula $\widehat{A}(\lambda) = \tr(AT^{\lambda,\nu})$, one could be tempted to study this Fourier transform for general trace class operators (without the assumption of radiality). Nevertheless, studying this transform is rather pointless, as the following fact shows:
\begin{lemma}
    Let $A \in \mathcal T^1(\mathcal A^2)$. Then  
    \begin{align*}
        \tr(AT^{\lambda,\nu}) = \tr(\radop{A}T^{\lambda,\nu}), \quad \phi_\lambda \in \spheripd.
    \end{align*}
    In particular, the map sending $A \mapsto [\lambda \mapsto \tr(AT^{\lambda,\nu})]$ is not injective on $\mathcal T^1(\mathcal A^2)$.
\end{lemma}
\begin{proof}
    The proof follows from a similar argument as in the proof of Proposition \ref{prop:dual_of_radial_trace_class}.
\end{proof}
Of course, injectivity of the Fourier transform on $\mathcal T^1(\Berg)^K$ can be deduced from the Plancherel theorem. We fix this as a separate observation:
\begin{proposition}
    The Fourier transform $\traceopK \ni A \mapsto \widehat{A}$ is injective.
\end{proposition}
A consequence of the Riemann-Lebesgue lemma and the Plancherel theorem is, by using the standard tools from the complex interpolation method:
\begin{proposition}[Hausdorff-Young estimate for operators]
    Let $1 \leq p \leq 2$ and $q \geq 2$ be the conjugate exponent, $\frac{1}{p} + \frac{1}{q} = 1$. Then, for $A \in \mathcal T^p(\mathcal A^2)^K$ it holds true that:
    \begin{align*}
        \| \mathcal F_K(A)\|_{L^q(\R^+, \mu_\fa)} \leq \| A\|_{\mathcal T^p}.
    \end{align*}
\end{proposition}
There is also a version of the Riemann-Lebesgue lemma for the inverse spherical Fourier transform of operators:
\begin{proposition}[Inverse Riemann-Lebesgue for operators]
    Denote by $\mathcal F^{-1}: L^1(\R^+, \mu_\fa) \to \mathcal L(\Berg)^K$ the formal inverse of the operator Fourier transform, given by the expression
    \begin{align*}
        \mathcal F^{-1}(f) = \int_{\R^+}f(\lambda)T^{\lambda,\nu}~d\mu_\fa(\lambda).
    \end{align*}
    Then, $\mathcal F_K^{-1}$ maps $L^1(\R^+, \mu_\fa)$ to $\mathcal K(\mathcal A^2)^K$, the radial compact operators. Furthermore, $\| \mathcal F_K^{-1}(f)\|_{op}\leq \| f\|_{L^1}$.
\end{proposition}
\begin{proof}
    The norm estimate is clear from the definition of $\mathcal F_K^{-1}(f)$ as a Gelfand-Pettis integral. We only need to prove compactness. For this, observe that for $f \in C_c(\R^+)$ we have $\mathcal F_K^{-1}(f) \in \mathcal T^2(\mathcal A^2)^K$. This together with the norm estimate immediately implies that each $\mathcal F_K^{-1}(f)$ with $f \in L^1(\R^+, \mu_\fa)$ can be approximated by compact operators in operator norm.
\end{proof}
Just as with the Hausdorff-Young estimate for operators, one obtains an inverse result:
\begin{proposition}[Inverse Hausdorff-Young estimate]
    Let $1\leq p \leq 2$ and $q \geq 2$ the conjugate exponent. Then, $\mathcal F_K^{-1}$ maps $L^p(\R^+, \mu_\fa)$ continuously to $\mathcal T^q(\mathcal A^2)^K$, and further the following norm estimate holds true:
    \begin{align*}
        \| \mathcal F_K^{-1}(f)\|_{\mathcal T^q} \leq \| f\|_{L^p(\R^+, \mu_\fa)}.
    \end{align*}
\end{proposition}

\appendix

 \section{Gelfand pairs, commutative spaces and spherical functions}\label{app:Gelfandpairs}
  Here we collect the main facts about Gelfand pairs and spherical Fourier transform. For details, we point to \cite{Helgason62, Wolf07, vD09}.
Let $G$ be a locally compact group and let $K$ be a compact subgroup of $G$. Then the left and right translations of a function $f:G\to \C$ by an element $x\in G$, denoted by $\ell_xf$ and $r_xf$ respectively, are given by
$$(\actfL{x}{f})(y)=f(x^{-1}y)\quad \text{and}\quad (r_xf)(y)=f(yx), \quad y\in G.$$
 
For a function space $\cF$ of locally integrable functions on $G$, we denote by $\cF^K$ the space of {\it left}-$K$-invariant functions in $\cF$: 
$$\cF^K=\{f\in\cF\mid (\forall k\in K)\,
\actfL{k}{f}= f\},$$
and we let $\cF_K$ be the space of {\it right}-$K$-invariant functions in $\cF$
\[\cF_K = \{f\in\cF\mid (\forall k\in K) \, r_kf= f\}.\]
We define the following projections: 
\[f^K (x) =\int_K f (kx) \haarK{k},\quad  f_K (x) = \int_K f(xk)\haarK{k}\quad\]
and
\[\quad f^\sharp (x) =\int_K\int_K f(k_1xk_2)\haarK{k_1}\haarK{k_2}\]
where we have $\mu_K$ denote the normalized Haar measure on $K$. 
Then $f\mapsto f^K$ defines a projection $\cF\to \cF^K$, $f\mapsto f_K$ defines a projection
onto $\cF_K$ and $f\mapsto f^\sharp $ defines a projection onto the space of $K$-biinvariant functions in $\cF$.
If $\cF$ is a function space on $G/K$ then $f_K =f$ and $f^K = f^\sharp$.

When the group $G$ is not abelian, the convolutions on the group are not necessarily commutative.
\begin{definition}
    The pair $(G,K)$, where $G$ is a locally compact Hausdorff topological group and $K$ is a compact subgroup, is called a {\it Gelfand pair} and the space $G/K$ is called a {\it commutative
 space}, if $L^1(G/K)^K$ is commutative.
\end{definition}

If $(G,K)$ is a Gelfand pair, then $G$ is unimodular. The \textit{spherical Fourier transform} of functions in $L^1(G/K)^K$ is given by the Gelfand transform of the commutative Banach $*$-algebra $L^1(G/K)^K$. Now we turn to the special case $G=\SU(n,1)$ and $K=\U(n)$, and observe that $\B^n=G/K$ is a commutative space. 

\begin{lemma} \label{conv} Let $f$ be a function on $G=\SU(n,1)$ and and let $g$ be a function on $\nball$. Then the following holds:
\begin{enumerate}
\item If $g$ is left $K$-invariant then $f*g = f_K *g$. 
\item If $f$ is right $K$-invariant then $f*g = f*g^K$. 
\item If $g$ is right invariant, then so is $f*g$.
\item If $f$ is left $K$-invariant, then $f*g  $ is left $K$-invariant.
\item If $g$ is right $K$-invariant and
$f$ left $K$-invariant, then $f*g$ is  $K$-bi-invariant.
 \end{enumerate}
 \end{lemma} 
 
 \begin{proof} This is well known and follows by a simple change of variable.
 \end{proof}

\begin{corollary}\label{Cor:32} Let $G=\SU(n,1)$. The convolution $L^1(G)\times \Lpinv$, $(f,g)\mapsto f*g$ leads to 
a $*$-representation of $L^1 (\SUn )$ on $\Lpinv$. If $f$ and $g$ are $K$-bi-invariant, then
$f*g\in \LpinvK$.
\end{corollary}

Now the following well-known fact follows, establishing that $(G,K)=(\SU(n,1), \U(n))$ is a Gelfand pair.

\begin{lemma}\label{lem:commutativity_convolutions}
    Let $f,g\in \LoneinvK$. Then $f\ast g=g\ast f$. Moreover, the algebra $\LoneinvK$ is a commutative Banach algebra.
\end{lemma}

\begin{proof}
    Considering $f$ and $g$ as functions on $G$, for $x\in G$, we have
    \begin{align*}
        (f\ast g) (x)&= \int_G f(y)g(y^{-1}x) \haar{y}\\
        &=  \int_G f(xy)g(y^{-1}) \haar{y}\quad \text{(by a change of variable)}\\
        &=\int_G f(y^{-1}x^{-1}) g(y) \haar{y} \quad \text{(by Lemma \ref{lem:radial_functions})}\\
        &= (g\ast f)(x^{-1})\\
        &= (g\ast f) (x)
    \end{align*}
    by Lemma \ref{lem:radial_functions} as $g\ast f$ is $K$-bi-invariant.
\end{proof}

\subsection{Spherical functions}
The spherical functions that we introduce below play the role of characters in abelian groups, in the context of commutative spaces.

\begin{definition} Let $(G,K)$ be a Gelfand pair.
 \begin{enumerate}
 \item A continuous non-zero function $\phi : G\to \C$ is a {\it spherical function} if for
 all $x,y\in G$ we have
 \[\int_K \phi (xky) d\mu_K(k) =\phi (x) \phi (y) .\]
 \item Let $X$ be a set. A kernel $K : X\times X\to \C$ is {\it positive definite} if for all $n\in \N$, $c_1,\ldots c_n\in \C$ and
 $x_1,\ldots , x_n\in K$ we have
 \[\sum_{j,k} c_j \bar c_k K(x_k,x_j) \ge 0 .\]
 If $G$ acts on $X$ then the kernel $K$ is $G$-quasi-invariant if there exists a cocycle $\sigma: G\times X\to \C$ such that
 $K(a\cdot x, a\cdot y ) =\sigma (a,x)\overline{\sigma (a,y)} K(x,y)$. The kernel $K$ is $G$-invariant if
 $\sigma=1$.
 \item A continuous function on $G$ is {\it positive definite} if  the $G$-invariant kernel $(a,b)\mapsto \phi (b^{-1}a)$ is
 positive definite. In that case we have $\phi (x) = \phi^*(x) = \overline{\phi (x^{-1})}$. 
 \end{enumerate}
 \end{definition}

  One easily checks that the spherical functions are $K$-bi-invariant. We denote by $\spf$ the set of all spherical functions on $G/K$ and by $\bddspheri$ we denote the set of all bounded spherical functions on $G/K$. Let
  $\pdsf$ be the set of all positive 
 definite spherical functions on $G/K$.

As an immediate consequence of the definition, we obtain the following:
\begin{lemma}\label{lemma:bddspherical_bound_one}
    Let $\phi \in \bddspheri$. Then, $|\phi|$ is bounded by $1$.
\end{lemma}
\begin{proof}
   Let $0 < \varepsilon < \| \phi\|_\infty$ and $y \in G$ such that $|\phi(y)| \geq \| \phi\|_\infty - \varepsilon$. Then, by the functional equation for spherical functions, we see that for each $x \in G:$
    \begin{align*}
        |\phi(x)| &\leq \frac{1}{|\phi(y)|} \int_K |\phi(xky)|~d\mu_K(k)\\
        &\leq \frac{\| \phi\|_\infty}{\| \phi|_\infty - \varepsilon} \int_K 1~d\mu_K(k) = \frac{\| \phi\|_\infty}{\| \phi|_\infty - \varepsilon}.
    \end{align*}
    Since $\varepsilon > 0$ was arbitrary, we see that $|\phi(x)|\leq 1$ for each $x \in G$.
\end{proof}

By \cite[Theorem 8.2.7]{Wolf07}, we have the following:
\begin{theorem} Let $(G,K)$ be a Gelfand pair, then the spectrum of $L^1(G/K)^K$ is given by
$\bddspheri$ where the isomorphism is given by
\[\chi_\phi (f) = \int_G f(x) \phi(x^{-1}) d\mu_G(x) , \quad f\in L^1(G/K).\]
\end{theorem}

We also recall that the set of positive spherical functions corresponds bijectively
to the set of irreducible unitary $K$-spherical representations of $G$.
\begin{theorem}[Theorem 8.4.8 in \cite{Wolf07}]\label{thm: cyclic_vec}
    We have the following:
    \begin{enumerate}
        \item Let $\phi\in \pdsf$. Then there exists (upto unitary equivalence) an irreducible unitary representation $(\rho, H_\rho)$ and a cyclic unit vector $u\in H_\rho$ such that $\phi(x)=\ip{u}{\rho(x)u}$ for all $x\in G$. Moreover, $u$ is a $K$-fixed vector, i.e. $\rho(k)u=u$ for all $k\in K$, and $u$ spans the space $H^K_\rho$ of all $K$-fixed vectors in $H_\rho$.
        \item Let $(\rho,H_\rho)$ be an irreducible unitary representation of $G$ s.t. $H^K_\rho$ is spanned by a single unit vector $u$. Then the function $\phi:\G\to \C$ defined by $\phi(x)=\ip{u}{\rho(x)u}$, $x\in G$, is a positive definite spherical function for $(G,K)$.
    \end{enumerate}
\end{theorem}

\subsection{Spherical Fourier transform}\label{subsec:sphericalFT}
The \textit{spherical Fourier transform} for the Gelfand pair $(G,K)$ is the map $\Gamma:f\mapsto\hat{f}$ given by
\begin{align}\label{eq:spherical_Fourier}
    \hat{f}(\phi)=\int_G f(x)\phi(x^{-1}) \haar{x}, \quad \phi\in \bddspheri.
\end{align}

We endow $\bddspheri$ with the weak topology induced by the maps $\{\hat{f}\mid f\in L^1(G/K)\}$.
We have that $\pdsf$ is locally compact w.r.t. the subspace topology of $\spf$ and the closure of $\spheripd=\pdsf$ is either $\overline{\pdsf}=\pdsf$ or $\overline{\pdsf}=\pdsf\cup\{0\}$. We proceed to discuss the following properties of the spherical Fourier transform. We point to \cite[section 9.2]{Wolf07} for the Riemann-Lebesgue lemma, \cite[Theorem 9.4.1]{Wolf07} for the Fourier inversion theorem, and \cite[Theorem 9.5.1]{Wolf07} for the Plancherel theorem.

\begin{theorem}[Riemann-Lebesgue lemma]\label{thm:ReimannLeb_functions}
    If $f\in L^1(G/K)^K$, then $\hat{f}\in C_0(\spf)$ and
    then $\hat{f|}_{\spheripd}\in C_0(\pdsf)$.
\end{theorem}

\begin{theorem}[Fourier inversion]\label{thm:Finversion_functions}
     There is a unique measure $\mu_P$ on $\pdsf$, called the Plancherel measure s.t. for any bounded function $f\in L^1(G/K)$, its spherical Fourier transform $\hat{f}\in L^1(\pdsf,\mu_P)$ and 
     \begin{align}
         f(x)=\int_{\cP} \hat{f}(\phi) \phi(x) \planch{\phi}, \quad x\in G.
     \end{align}
\end{theorem}

\begin{theorem}[Plancherel theorem]\label{thm:Plan1}
The spherical Fourier transform $f\mapsto \wf$ extends from $L^1(G/K)^K\cap L^2(G/K)^K$ to an isometric
isomorphism $L^2(G/K)^K \simeq L^2(\pdsf ,\mu_P)$.
\end{theorem}

\bibliographystyle{alpha}
\bibliography{References}


\begin{multicols}{2}
\noindent
Vishwa Dewage\\
\href{vdewage@wm.edu}{\Letter ~vdewage@wm.edu}
\\
\noindent
Department of Mathematics\\
College of William \& Mary\\
200 Ukrop Way\\
Williamsburg
VA 23185\\
USA

\noindent
Robert Fulsche\\
\href{fulsche@math.uni-hannover.de}{\Letter ~fulsche@math.uni-hannover.de}
\\
\noindent
Institut f\"{u}r Analysis\\
Leibniz Universit\"at Hannover\\
Welfengarten 1\\
30167 Hannover\\
GERMANY\\ 
\end{multicols}

\noindent Gestur \'{O}lafsson\\
\href{olafsson@lsu.edu}{\Letter ~olafsson@lsu.edu}
\\
\noindent
Department of Mathematics\\
Louisiana State University\\
Baton Rouge LA 70803\\
USA

\end{document}